\documentclass[final,1p,times,11pt]{elsarticle}
\usepackage{geometry}
\usepackage{graphicx,subfigure,epsfig,multirow}
\usepackage{epstopdf}
\usepackage{amsmath,mathrsfs,amssymb,amsthm,dsfont,bm,amsfonts}     
\usepackage{hyperref}
\usepackage{lipsum}
\usepackage{float}
\usepackage{array, caption, threeparttable}
\usepackage{booktabs}
\usepackage{bbding}
\usepackage{diagbox}
\usepackage[misc]{ifsym}
\usepackage{multicol}
\usepackage{cases}

\newtheorem{alg}{Algorithm}[section]
\newtheorem{theorem}{Theorem}[section]
\newtheorem{assumption}{Assumption}[section]

\theoremstyle{definition}

\newtheorem{example}{Example}[subsection]
\allowdisplaybreaks[4]
\theoremstyle{remark}
\newtheorem{remark}{\bf{Remark}}

\numberwithin{equation}{section}
\numberwithin{figure}{section}
\numberwithin{table}{section}
\usepackage{etoolbox}
\makeatletter
\def\ps@pprintTitle{%
   \let\@oddhead\@empty
   \let\@evenhead\@empty
   \def\@oddfoot{\centerline{\thepage}}%
   \let\@evenfoot\@oddfoot}
\makeatother

\begin{document}
\title{A novel energy-conservation Crank-Nicolson finite element method for generalized Klein-Gordon-Zakharov equations }

\author[a]{Xuemiao Xu}
\ead{xu_xuem@163.com}

\author[b]{Maosheng Jiang}
\ead{msjiang@qdu.edu.cn}

\author[a]{Jiansong Zhang\corref{zhang}}
\ead{jszhang@upc.edu.cn}

\author[c]{Jiang Zhu}
\ead{jiang@lncc.br}

\address[a]{Department of Applied Mathematics, China University of Petroleum, Qingdao 266580, China}
\address[b]{School of Mathematics and Statistics, Qingdao University,  Qingdao 266071, China}

\address[c]{Laborat\'orio Nacional de Computa\c{c}\~ao Cient\'{\i}fica, MCTI 
        Avenida Get\'ulio Vargas 333, 25651-075 Petr\'opolis, RJ, Brazil}

\cortext[zhang]{Corresponding author. Jiansong Zhang}

\begin{abstract}
This article focuses on an energy-conservation Galerkin finite element method (FEM) for the generalized Klein-Gordon-Zakharov (KGZ) equations. This method combines the bilinear finite element method for spatial discretization with the Crank-Nicolson (CN) scheme for temporal discretization, thereby guaranteeing exact conservation of the discrete energy functional. A rigorous theoretical analysis  is devoted to deriving  error bounds for the fast-time-scale electronic field $u$ and the ion density deviation $\varphi$. By systematically integrating interpolation estimates, Ritz projection, and a postprocessing technique,  the superclose error estimates and global superconvergence  are established for $u$ in the $H^1$-norm, even under weakened regularity assumptions on the exact solution. Concurrently, we prove $H^1$-norm superconvergence for the auxiliary variable $\phi$  ($-\Delta\phi = \varphi_t$) and optimal-order $L^2$-norm error estimates for  the auxiliary variable $p$ ($p=u_t$) and $\varphi$. Numerical examples are provided to confirm theoretical results.

\end{abstract}
\begin{keyword}
	Superclose and superconvergence, Energy-conservation,  Interpolation and Ritz projection, Generalized KGZ equations.
\end{keyword}

\maketitle

\section{Introduction}
As a fundamental mathematical model in plasma physics, the Klein–Gordon–Zakharov equations accurately depict the nonlinear coupling between high-frequency Langmuir waves and low-frequency ion-acoustic waves, see \cite{baskonus2019new,nestor2020new}. In this paper, we consider a novel numerical method for simulating the following generalized KGZ equations (\cite{zakharov1972collapse})
	\begin{equation}
		\label{eq1.1}
		\left\{\begin{aligned}
			&(\mathrm{a})\quad u_{tt}-\Delta u+u+\varphi u+|u|^2u=0, &&(\mathbf{x},t)\in\Omega\times J , \\
			&(\mathrm{b})\quad \varphi_{tt}-\Delta \varphi-\Delta |u|^2=0, &&(\mathbf{x},t)\in\Omega\times J , \\
			&(\mathrm{c})\quad u(\mathbf{x},t)=0,\quad \varphi(\mathbf{x},t)=0  &&(\mathbf{x},t)\in\partial\Omega\times J, \\
			&(\mathrm{d})\quad u(\mathbf{x},0)=u_0(\mathbf{x}),  \quad	u_t(\mathbf{x},0)=u_1(\mathbf{x}), 	&&\mathbf{x}\in\Omega,\\
			&(\mathrm{e})\quad \varphi(\mathbf{x},0)=\varphi_0(\mathbf{x}), \quad \varphi_t(\mathbf{x},0)=\varphi_1(\mathbf{x}), 	&&\mathbf{x}\in\Omega,
		\end{aligned}
		\right.
	\end{equation}
where $\Omega \subset \mathbb{R}^d\ (d=2,3)$ is a bounded convex domain with boundary $\partial\Omega$; $J = (0, T]\,\,(0<T)$ represents the time interval; the complex-valued unknown function $u(\mathbf{x}, t)$ captures the fast-time-scale electronic component of the high-frequency electric field, and the  real-valued unknown function $\varphi(\mathbf{x}, t)$ describes the low-frequency deviation of the ion density from its equilibrium state; the initial data $u_0(\mathbf{x})$, $u_1(\mathbf{x})$, $\varphi_0(\mathbf{x})$, and $\varphi_1(\mathbf{x})$ are given sufficiently smooth functions; physically, the cubic nonlinear term $|u|^2 u$ accounts for the nonlinear self-interaction of Langmuir waves. 

The KGZ system \eqref{eq1.1}  possesses an important structural invariant: the exact conservation of total energy \cite{bao2016uniformly,gao2017galerkin}, which reads as
\begin{equation}
	\frac{dE}{dt} = 0, \quad t \in (0, T],\nonumber
\end{equation}
where the energy functional $E(t)$ is given as follows
\begin{equation}
	E(t) = \int_{\Omega} \left( |\nabla u|^2 + |u|^2 + |u_t|^2 + \frac{1}{2}|\varphi|^2 + \frac{1}{2}|\nabla \phi|^2 + \frac{1}{2}|u|^4 + \varphi|u|^2 \right) d\mathbf{x},\nonumber
\end{equation}
with $\varphi_t = \Delta \phi$.

Over the past decades, the generalized KGZ equations have attracted sustained theoretical analysis and numerical simulation. For instance, the well-posedness and the existence of global smooth solutions for the KGZ equations have been established in  \cite{boling1995global,adomain1997non,ozawa1999well,ismail20101,kumar2021some}. Owing to the strong nonlinear coupling intrinsic to the generalized KGZ system,  the derivation of exact analytical solutions is precluded for generic initial and boundary conditions. Consequently, The development of efficient and structure-preserving numerical algorithms, particularly those that exactly conserve the total energy functional, has attracted considerable scholarly attention. Among classical spatial discretization methods, the finite difference methods (FDMs) have been widely employed for its flexibility in designing structure-preserving schemes. A series of conservative implicit, explicit, and semi-explicit compact difference schemes have been systematically developed and rigorously analyzed in  \cite{wang2007conservative,wang2013convergence,chen2012numerical}.	For complex scenarios involving different parameter limit regimes, two conservative fourth-order compact finite difference schemes specifically tailored for the subsonic limit regime were investigated, see \cite{li2024analysis}. Alternatively, FEMs offer distinct advantages in accommodating complex geometries and enabling rigorous, problem-adapted error estimation. Galerkin FEMs have been developed for generalized KGZ equations  \cite{gao2017galerkin}. Moreover, for KGZ systems with power-law nonlinearities, a bilinear Galerkin FEM was comprehensively analyzed in \cite{shi2022high}, establishing superclose and global superconvergence results; subsequently, an $H^1$-Galerkin mixed FEM was developed in \cite{wang2024superconvergence}, yielding optimal-order superconvergence estimates for both primary and auxiliary variables in their respective norms.

To address the severe numerical challenges arising from highly oscillatory solutions in specific physical regimes, e.g. the high-plasma-frequency limit, spectral methods and their variants have been systematically introduced. A uniformly accurate multiscale time integrator spectral method was devised  in \cite{bao2016uniformly} to effectively suppress the temporal high-frequency oscillations. Furthermore, several high-fidelity meshless approximation frameworks have been rigorously investigated: the radial basis function partition of unity (RBF-PU) method \cite{barikbin2025radial}, the meshless collocation method based on barycentric rational interpolation \cite{orucc2022numerical}, and the trigonometric quintic and exponential B-spline schemes specifically designed for simulating strong Langmuir turbulence \cite{khater2021strong}.  Collectively, these advances have substantially broadened the scope and reliability of numerical methods for high-dimensional and generalized KGZ systems.

Recently, linearized structure-preserving algorithms based on auxiliary variables have gained considerable traction, as they rigorously preserve the system's energy-dissipative or energy-conservative structure at the discrete level while substantially enhancing computational efficiency for strongly nonlinear systems, see \cite{jiang1,jiang2,jiang3}. In  \cite{Li2024new},  a class of high-order energy-preserving algorithms based on the supplementary variable method (SVM) was introduced by reformulating the system into an equivalent optimization problem. Moreover, the scalar auxiliary variable (SAV) and exponential scalar auxiliary variable (ESAV) approaches, combined with conservative nonstandard FDM, were employed to construct linearly unconditionally stable schemes that  preserve modified energy conservation \cite{guo2023energy,fazayel2026numerical}. By combining the quadratic auxiliary variable (QAV) approach with symplectic Runge-Kutta methods, high-order structure-preserving schemes  have also been successfully developed for the generalized  KGZ system \cite{guo2025high}. It is worth noting that, driven by the discovery of anomalous diffusion in complex plasma environments, the classical KGZ system has been generalized to the fractional calculus framework. For spatial or bi-fractional KGZ equations, a series of energy-conserving linear implicit difference schemes and time second-order operator splitting methods have been proposed in \cite{xie2019analysis,xie2021linear,xie2022stable}. In addition, explicit energy-conserving discretizations have been designed in \cite{martinez2021implicit,martinez2019theoretical}, while a high-order structure-preserving scheme for the space-fractional system was presented via an exponential Fourier collocation approach combined with a shifted Lubich difference operator in \cite{li2026high}.

Notwithstanding the extensive body of theoretically rigorous and well-established advancements concerning the numerical approximation of classical, generalized, and fractional KGZ equations, a conspicuous research gap continues to prevail within the existing literature: there remains a notable absence of fully discrete CN-type discretization schemes that are capable of precisely preserving energy invariants while rigorously retaining the inherent Hamiltonian structure of the KGZ system across both temporal and spatial discretization procedures. It is widely acknowledged that any improper discretization procedure that fails to faithfully conserve the intrinsic structural characteristics possessed by continuous governing equations may inevitably deteriorate both the mathematical well-posedness and long-time computational fidelity of the resultant discrete numerical models, thereby highlighting the indispensable significance of structure-preserving numerical discretization techniques within computational mathematical research \cite{li1995finite}. In view of the aforementioned theoretical necessity and research motivation, the present study is dedicated to the construction of an energy-conserved CN numerical scheme for the coupled nonlinear KGZ system formulated in Eq. (\ref{eq1.1}). Within the theoretical framework of conforming bilinear Galerkin finite element formulations, we further conduct a rigorous analytical derivation to establish superclose properties and global superconvergence estimates under the $H^1$-norm. To facilitate the implementation of the subsequent theoretical error analysis, a set of auxiliary variables denoted by $p = u_t$ and $\phi$ are introduced throughout this work, wherein the intermediate variable $\phi$ is implicitly constrained by the elliptic partial differential relation $\Delta\phi = \varphi_t$. The major scientific contributions elaborated within this paper can be explicitly summarized into three comprehensive aspects: (i) in contrast to the methodological limitations manifested in the prior literature presented in \cite{shi2022high}, the presently proposed fully discrete energy-conservation CN scheme effectively overcomes the inherent computational difficulties arising from strongly nonlinear coupling terms embedded in the KGZ system, which further enables the establishment of a structurally consistent and computationally robust numerical framework for nonlinear wave simulations; (ii) rigorous mathematical proofs are comprehensively performed to verify the global $H^1$-norm superconvergence behavior with second-order accuracy $O(h^2 + \tau^2)$ for the primary physical field variables $u$ and $\phi$, while optimal $L^2$-norm error estimates possessing identical convergence orders are simultaneously derived for the auxiliary variables $p$ and $\varphi$; (iii) by virtue of the synergistic combination between Ritz projection operators and conventional Lagrange interpolation techniques, the spatial regularity prerequisites imposed on analytical solutions are substantially ameliorated in comparison with traditional interpolation-based error analysis strategies. More specifically, the minimal smoothness requirement is rationally reduced from $u_t, p \in \mathcal{H}^3(\Omega)$ and $\varphi \in H^3(\Omega)$ to a milder regularity condition satisfying $u_t, p \in \mathcal{H}^2(\Omega)$ and $\varphi \in H^2(\Omega)$. Ultimately, a collection of representative numerical experiments are carried out to substantiate the rationality of the derived theoretical conclusions, whereby the computational reliability and structural preservation capability of the developed numerical methodology are validated through extensive numerical simulations concerning wave dynamical evolutions and nonlinear interaction phenomena inherent in two-dimensional and three-dimensional KGZ models.

In what follows, we denote by $W^{m,p}(\Omega)$ the classical Sobolev spaces for $1 \leq p \leq \infty$, endowed with the standard norm $\|\cdot\|_{W^{m,p}}$ and seminorm $|\cdot|_{W^{m,p}}$. For brevity, we adopt the following simplified notations: $L^p(\Omega)$, $H^m(\Omega)$, and $\| \cdot \|_m$ for $W^{0,p}(\Omega)$, $W^{m,2}(\Omega)$, and $\| \cdot \|_{m,2}$, respectively. Let $\gamma, \omega : \Omega \to \mathbb{C}$ be arbitrary functions. The associated scalar inner product is then given by
\begin{equation}
(\gamma, \omega) = \int_\Omega \gamma \hat{\omega} \text{ dxdy}, \nonumber
\end{equation}
with $\hat{\omega}$ representing the complex conjugate. Let $L^p(J;Z)$ be the time dependent Banach space. For $1\leq p<\infty$, define 
\begin{equation}
\|g\|_{L^p(J;Z)}=\bigg(\int_J\|g(\cdot,t)\|_Z^p\text{ dt} \bigg)^{\frac{1}{p}}.\nonumber
\end{equation}
For $p=\infty$, set $\|g\|_{L^\infty(J;Z)}=\text{sup}_{t\in J}\|g(\cdot,t)\|_Z$. Consider a complex-valued function $\psi=\psi_1+i\psi_2$. Define
\begin{equation}
\mathcal{L}^2(\Omega)=\{\psi_1+i\psi_2:\ \psi_1,\psi_2\in L^2(\Omega)\},\quad
\mathcal{H}^m(\Omega)=\{\psi_1+i\psi_2:\ \psi_1,\psi_2\in H^m(\Omega)\}.\nonumber
\end{equation}
Throughout this article, $C$ denotes a generic positive constant independent of mesh size $h$ and time increment $\tau$.

The remainder of this paper is organized as follows. Section 2 introduces the necessary preliminaries, and presents the energy-conserved CN scheme, followed by a rigorous proof of its unique solvability. In Section 3, we derive global supercloseness and superconvergence estimates in the $H^1$-norm for the primary variables. Section 4 provides comprehensive numerical experiments to validate the theoretical results—including convergence rates, energy preservation, and long-term stability.  Finally, concluding remarks are offered in Section 5.

\section{Energy conservation Crank-Nicolson type fully discrete scheme}

 To illustrate our method, we introduce two auxiliary functions $p$ and $\phi$:
 \[
  p = u_t,\quad 
   \Delta\phi = \varphi_t. 
  \]
 Consequently, the  KGZ system (\ref{eq1.1}) can be rewritten into the subsequent equivalent formulation:
 \begin{equation}\label{eq2.1}
 	\left\{\begin{aligned}
 		&p_{t}-\Delta u+u+\varphi u+|u|^2u=0, &&(\mathbf{x},t)\in\Omega\times J , \\
 		&p-u_t=0, &&(\mathbf{x},t)\in\Omega\times J , \\
 		&\phi_{t}-\varphi-|u|^2=0, &&(\mathbf{x},t)\in\Omega\times J , \\
 		&\Delta\phi-\varphi_t=0, &&(\mathbf{x},t)\in\Omega\times J , \\
 		&u(\mathbf{x},t)=0,\quad  p(\mathbf{x},t)=0,\quad  \varphi(\mathbf{x},t)=0,\quad  \phi(\mathbf{x},t)=0,  &&(\mathbf{x},t)\in\partial\Omega\times J, \\
 		&u(\mathbf{x},0)=u_0(\mathbf{x}),  \quad 	u_t(\mathbf{x},0)=u_1(\mathbf{x}), 	&&\mathbf{x}\in\Omega,\\
 		&\varphi(\mathbf{x},0)=\varphi_0(\mathbf{x}), \quad \varphi_t(\mathbf{x},0)=\varphi_1(\mathbf{x}), 	&&\mathbf{x}\in\Omega.
 	\end{aligned}
 	\right.
 \end{equation} 
 The corresponding weak formulation of (\ref{eq2.1}) can be given as follows: find $(u,p,\phi,\varphi)\in \mathcal{H}_0^1(\Omega)\times \mathcal{H}_0^1(\Omega)\times H_0^1(\Omega)\times H_0^1(\Omega)$,  such that 
\begin{equation}\label{eq1.2}
		\left\{\begin{aligned}
			&(p,v_1)=(u_t,v_1),&&\forall v_1\in\mathcal{L}^2(\Omega), \\
			&(p_t,v_2)+(\nabla u,\nabla v_2)+(u,v_2)+(\varphi u,v_2)+(|u|^2u,v_2)=0, &&\forall v_2\in\mathcal{H}_0^1(\Omega), \\
			&(\nabla\phi,\nabla w_1)=-(\varphi_t,w_1), &&\forall w_1\in H_0^1(\Omega), \\
		 &  (\phi_t,w_2)-(\varphi,w_2)-(|u|^2,w_2)=0, &&\forall w_2\in L^2(\Omega),\\
		   	&u(\mathbf{x},0)=u_0(\mathbf{x}),   \quad 	p(\mathbf{x},0)=u_1(\mathbf{x}), 	&&\mathbf{x}\in\Omega,\\
		   & \varphi(\mathbf{x},0)=\varphi_0(\mathbf{x}),  \quad  \phi(\mathbf{x},0)=\varphi_0(\mathbf{x}),	&&\mathbf{x}\in\Omega,
		\end{aligned}
		\right.
\end{equation}
where $\mathcal{H}_0^1(\Omega)=\{\xi_1+i\xi_2\,|\,\xi_1,\xi_2\in H^1_0(\Omega)\}$.
 
For the purpose of characterizing the inherent conservative property of the considered coupled KGZ system, we define the continuous energy functional in the following form:
 \begin{equation}
 	E(t) = \|\nabla u\|_0^2 + \|u\|_0^2 + \|p\|_0^2 + \frac{1}{2}\|\varphi\|_0^2 + \frac{1}{2}\|\nabla\phi\|_0^2 + \frac{1}{2}\int_{\Omega}|u|^4 \, d\mathbf{x} + (\varphi, |u|^2).\nonumber
 \end{equation}
As theoretically validated and mathematically proven in \cite{guo2023energy}, the analytical solution $(u, p, \phi, \varphi)$ to the continuous governing system (\ref{eq2.1}) obeys the intrinsic energy conservation law, which yields the following differential conservation identity:
 \begin{equation}
 	\frac{d}{dt}E(t) = 0, \quad t \in [0, T].\nonumber
 \end{equation}
 
To facilitate the rigorous construction of the fully discrete numerical scheme, we first perform a standard domain discretization procedure on the computational domain $\Omega \subset \mathbb{R}^d$ with $d=2,3$. Specifically, the spatial domain is partitioned into a family of uniform meshes $\mathcal{T}_h=\{K\}$, where each mesh element corresponds to a rectangular cell for the two-dimensional case and a cuboid cell for the three-dimensional scenario, with $h$ denoting the spatial mesh size. Consistent with the spatial discretization framework presented in \cite{lin2006finite}, the conforming finite element spaces consisting of bilinear or trilinear basis functions are formally defined as follows:
\begin{equation}
	V_0^h=\{v_h\in C(\Omega);\ v_h|_K\in\mathcal{Q}_1(K),\ v_h|_{\partial\Omega}=0,\ \forall K\in T_h\},\nonumber
\end{equation}
where $\mathcal{Q}_1(K) = \text{span}\{1, x, y, xy\}$ for the two-dimensional case, and $\mathcal{Q}_1(K) = \text{span}\{1, x, y, z, xy, yz, zx, xyz\}$ for the three-dimensional case. For complex value function $\xi$, define the finite element space
\begin{equation}
	\mathcal{V}_0^h\overset{\cdot}{=}\{\xi_1+i\xi_2\,|\,\xi_1,\xi_2\in V_0^h\}.\nonumber
\end{equation}

To discretize the time interval $[0, T]$, we introduce a uniform temporal mesh characterized by the nodes $t_n = n\tau$ for $n = 0, 1, \dots, N$, where $\tau = T/N$ represents the fixed time step size. The half-integer time points are defined accordingly as $t_{n-1/2} = (n - 1/2)\tau$. Given a sufficiently smooth function $\varrho(\mathbf{x}, t)$, we denote its discrete evaluation at time $t_n$ by $\varrho^n = \varrho(\mathbf{x}, t_n)$ and at the midpoint by $\varrho^{n-1/2} = \varrho(\mathbf{x}, t_{n-1/2})$. For any discrete sequence $\{\varrho^n\}_{n=0}^N$, we define 
\begin{equation}
\partial_t \varrho^n = \frac{\varrho^n - \varrho^{n-1}}{\tau}, \quad \overline{\varrho^n} = \frac{\varrho^n + \varrho^{n-1}}{2}.
\end{equation}
	
Subsequently, we proceed to formulate the fully discrete finite element approximation procedure for the aforementioned coupled KGZ system.
\begin{alg}\label{alg1}
For given initial approximate values $(u_h^0,p_h^0,\phi_h^0,\varphi_h^0)\in \mathcal{V}_0^h\times\mathcal{V}_0^h \times V_0^h\times V_0^h$, 
seek  $(u_h^n,p_h^n,\phi_h^n,\varphi_h^n)\in \mathcal{V}_0^h\times\mathcal{V}_0^h \times V_0^h\times V_0^h$, $n=1,2,3,\ldots,N$, such that,
		\begin{equation}
		\label{eq1.3}
		\left\{\begin{aligned}
		&(\mathrm{a})\quad (\overline{p_h^n},v_{1h})=(\partial_tu_h^n, ,v_{1h}) ,&\forall v_{1h}\in\mathcal{V}_0^h,\\
		&(\mathrm{b})\quad	(\partial_tp_h^n,v_{2h})+(\nabla\overline{u_h^n},\nabla v_{2h}
			)+(\overline{u_h^n},v_{2h})+(\overline{\varphi_h^n}\,\overline{u_h^n},v_{2h})+(\frac{1}{2}(|u_h^n|^2+|u_h^{n-1}|^2)\overline{u_h^n},v_{2h})=0, &\forall v_{2h}\in\mathcal{V}_0^h, \\
		&(\mathrm{c})\quad	(\nabla\overline{\phi_h^n},\nabla w_{1h})=-(\partial_t\varphi_h^n,w_{1h}), &\forall w_{1h}\in V_0^h, \\
		&(\mathrm{d})\quad	(\partial_t\phi_h^n,w_{2h})- (\overline{\varphi_h^n},w_{2h})-(\frac{1}{2}(|u_h^n|^2+|u_h^{n-1}|^2),w_{2h})=0, &\forall w_{2h}\in V_0^h.
		\end{aligned}
		\right.
	\end{equation}
\end{alg}
\begin{remark}
In fact, to rigorously preserve the essential property  (\ref{eq2.9}), the standard Crank-Nicolson discretization is modified in Algorithm \ref{alg1}  by replacing the terms $|\overline{u_h^n}|^2\overline{u_h^n}$ and $|\overline{u_h^n}|^2$ with $\frac{1}{2}(|u_h^n|^2 + |u_h^{n-1}|^2)\overline{u_h^n}$ and $\frac{1}{2}(|u_h^n|^2 + |u_h^{n-1}|^2)$, respectively.
\end{remark}

Define the discrete energy functional corresponding to the developed numerical algorithm:
\[
E^n=\|u_h^n\|_1^2+\|u_h^n\|_0^2+\|p_h^n\|_0^2+\frac{1}{2}\|\varphi_h^n\|_0^2+\frac{1}{2}\|\phi_h^n\|_1^2+\frac{1}{2}\|u_h^n\|_{0,4}^4+(\varphi_h^n,|u_h^n|^2).
\]
It can be mathematically verified that the aforementioned coupled system (\ref{eq1.3}) inherently admits the subsequent energy conservation characteristic throughout the temporal evolution. 

\begin{theorem}
Suppose that the numerical solution $(u_h^n,p_h^n,\phi_h^n,\varphi_h^n)$ satisfies the aforementioned fully discrete finite element scheme  (\ref{eq1.3}). Then, the discrete energy functional $E^n$ obeys the exact conservation law over discrete temporal layers, i.e.,
\begin{equation}
E^N=E^{N-1}=\cdots=E^0.  \label{eq2.9}
\end{equation}
\end{theorem}
\begin{proof}

Let $v_{2h}=u_h^n-u_h^{n-1}$ and $w_{1h}=\phi_h^n-\phi_h^{n-1}$ in (\ref{eq1.3})(b)(c), and $v_{1h}=p_h^n-p_h^{n-1}$ and $w_{2h}=\varphi_h^n-\varphi_h^{n-1}$ in (\ref{eq1.3})(a)(d). Taking the real part yields
\[
\begin{aligned}
&\|p_h^n\|_0^2-\|p_h^{n-1}\|_0^2+\|u_h^n\|_1^2-\|u_h^{n-1}\|_1^2+\|u_h^n\|_0^2-\|u_h^{n-1}\|_0^2
\\&+\frac{1}{2}\int_\Omega\bigg(|u_h^n|^2\varphi_h^n+|u_h^n|^2\varphi_h^{n-1}-|u_h^{n-1}|^2\varphi_h^n-|u_h^{n-1}|^2\varphi_h^{n-1}\bigg)d\mathbf{x}+\frac{1}{2}\|u_h^n\|_{0,4}^4-\frac{1}{2}\|u_h^{n-1}\|_{0,4}^4=0, 
\end{aligned}
\]
and
\[
\frac{1}{2}\|\phi_h^n\|_1^2-\frac{1}{2}\|\phi_h^n\|_1^2+\frac{1}{2}\|\varphi_h^n\|_0^2-\frac{1}{2}\|\varphi_h^n\|_0^2+\frac{1}{2}\int_{\Omega}\bigg(|u_h^n|^2\varphi_h^n-|u_h^n|^2\varphi_h^{n-1}+|u_h^{n-1}|^2\varphi_h^n-|u_h^{n-1}|^2\varphi_h^{n-1}\bigg)d\mathbf{x}=0.
\]
From the above two equations, it follows that 
\begin{equation}
\begin{aligned}
&\|u_h^n\|_1^2+\|u_h^n\|_0^2+\|p_h^n\|_0^2+\frac{1}{2}\|\varphi_h^n\|_0^2+\frac{1}{2}\|\phi_h^n\|_1^2+\frac{1}{2}\|u_h^n\|_{0,4}^4+(\varphi_h^n,|u_h^n|^2)\\
=& \|u_h^{n-1}\|_1^2+\|u_h^{n-1}\|_0^2+\|p_h^{n-1}\|_0^2+\frac{1}{2}\|\varphi_h^{n-1}\|_0^2+\frac{1}{2}\|\phi_h^{n-1}\|_1^2+\frac{1}{2}\|u_h^{n-1}\|_{0,4}^4+(\varphi_h^{n-1},|u_h^{n-1}|^2).
\end{aligned}
\end{equation}
Hence, \eqref{eq2.9} can be readily derived from the above equation, which further manifests that the newly constructed numerical scheme precisely retains the discrete energy balance at each temporal level.
\end{proof}

Subsequently, we proceed to rigorously investigate and establish the unique solvability of the fully discrete numerical algorithm.
\begin{theorem}
For any $n= 1, 2, 3, \ldots, N$, Algorithm \ref{alg1} admits a unique solution $
(u_h^n,p_h^n,\phi_h^n,\varphi_h^n)\in \mathcal{V}_0^h\times\mathcal{V}_0^h \times V_0^h\times V_0^h$.
\end{theorem}
\begin{proof}
we define an operator $\mathcal{T}:(\mathcal{V}_0^h)^2\times(V_0^h)^2\rightarrow(\mathcal{V}_0^h)^2\times(V_0^h)^2$, such that for given $u_h^{n-1}$,  $p_h^{n-1}$, $U_h$, $P_h\in\mathcal{V}_0^h$, $\varphi_h^{n-1}$, $ \phi_h^{n-1}$, $X_h$,  $Y_h\in V_0^h$, denote $(\hat{U}_h, \hat{P}_h, \hat{X}_h, \hat{Y}_h)=\mathcal{T}(U_h, P_h, X_h, Y_h)$ such that
\begin{equation}
	\label{eq1.4}
\left\{\begin{aligned}
&(\mathrm{a})\quad \frac{1}{\tau}(\hat{U}_h-u_h^{n-1},v_{1h})=\frac{1}{2}(\hat{P}_h^n+p_h^{n-1},v_{1h}), &\forall v_{1h}\in\mathcal{V}_0^h , \\
&(\mathrm{b})\quad \frac{1}{\tau}(\hat{P}_h-p_h^{n-1},v_{2h})+\frac{1}{2}(\nabla(\hat{U}_h+u_h^{n-1}),\nabla v_{2h})+\frac{1}{2}((\hat{U}_h+u_h^{n-1}), v_{2h})&
\\	&\qquad\quad=-\frac{1}{4}(((U_h)^2+(u_h^{n-1})^2)(U_h+u_h^{n-1}),v_{2h})-\frac{1}{4}((X_h+\varphi_h^{n-1})(U_h+u_h^{n-1}),v_{2h}), &\forall v_{2h}\in\mathcal{V}_0^h, \\	
&(\mathrm{c})\quad \frac{1}{\tau}(\hat{X}_h-\varphi_h^{n-1},w_{1h})=\frac{1}{2}(\nabla(\hat{Y}+\phi_h^{n-1}),\nabla w_{1h}), &\forall w_{1h}\in V_0^h, \\	&(\mathrm{d})\quad \frac{1}{\tau}(\hat{Y}-\phi_h^{n-1}, w_{2h})-\frac{1}{2}(\hat{X}_h+\varphi_h^{n-1}, w_{2h})=\frac{1}{2}((U_h)^2+(u_h^{n-1})^2,w_{2h}), &\forall w_{2h}\in V_0^h.
	\end{aligned}
	\right.
\end{equation}

Let $\{\hat{l}_i\}_{i=1}$ and $\{l_i\}_{i=1}$ be a group basis of $\mathcal{V}_0^h$ and $ {V}_0^h$, respectively, $\hat{U}_h=\sum\limits_{i=1}^k\hat{u}_i\hat{l}_i$, $\hat{P}_h=\sum\limits_{i=1}^k\hat{p}_i\hat{l}_i$, $\hat{X}_h=\sum\limits_{i=1}^k\hat{x}_il_i$, $\hat{Y}_h=\sum\limits_{i=1}^k\hat{y}_il_i$. We can rewrite (\ref{eq1.4}) in matrix form:
\begin{equation}
	\label{eq1.5}
	\left\{\begin{aligned}
&(\mathrm{a})\quad  2\hat{A}\hat{\boldsymbol{U}}=\tau \hat{A}\hat{\boldsymbol{P}}+\boldsymbol{Z}^1, \\
&(\mathrm{b})\quad  2\hat{A}\hat{\boldsymbol{P}}+\tau \hat{B}\hat{\boldsymbol{U}}+\tau\hat{A}\hat{\boldsymbol{U}}=\boldsymbol{Z}^2,\\
&(\mathrm{c})\quad 2A\hat{\boldsymbol{X}}=\tau B\hat{\boldsymbol{Y}}+\boldsymbol{Z}^3, \\
&(\mathrm{d})\quad  2A\hat{\boldsymbol{Y}}-\tau A\hat{\boldsymbol{X}}=\boldsymbol{Z}^4,
	\end{aligned}
	\right.
\end{equation}
where 
\[ 
\begin{aligned}
&\hat{A}=(\hat{l}_i, \hat{l}_j)_{k\times k}, \quad \hat{B}=(\nabla\hat{l}_i,\nabla\hat{l}_j)_{k\times k}, \\
&A=(l_i,l_j)_{k\times k},\quad B=(\nabla l_i,\nabla l_j)_{k\times k},
 \end{aligned}
\]
 \[
   \begin{aligned}
 &\hat{\boldsymbol{U}}=(\hat{u}_1,\hat{u}_2,\ldots,\hat{u}_k)^T, \quad 
 \hat{\boldsymbol{P}}=(\hat{p}_1,\hat{p}_2,\ldots,\hat{p}_k)^T,\\
 & \hat{\boldsymbol{X}}=(\hat{x}_1,\hat{x}_2,\ldots,\hat{x}_k)^T, \quad 
  \hat{\boldsymbol{Y}}=(\hat{y}_1,\hat{y}_2,\ldots,\hat{y}_k)^T, 
 \end{aligned}
  \]
 \[
 \boldsymbol{Z^1}=(z^1_j)_{k\times1}, \quad 
 \boldsymbol{Z^2}=(z^2_j)_{k\times1}, \quad 
 \boldsymbol{Z^3}=(z^3_j)_{k\times1}, \quad 
 \boldsymbol{Z^4}=(z^4_j)_{k\times1},
 \]
  \[
  \begin{aligned}
  z^1_j=&\tau(p_h^{n-1},\hat{l}_j)+2(u_h^{n-1},\hat{l}_j),
\\
    z^2_j=&(-\frac{\tau}{2}((U_h)^2+(u_h^{n-1})^2)(U_h+u_h^{n-1})+2p_h^{n-1}-\tau u_h^{n-1},\\
    &-\frac{\tau}{2}((X_h+\varphi_h^{n-1})(U_h+u_h^{n-1}),\hat{l}_j)-\frac{1}{2}(\nabla u_h^{n-1},\nabla \hat{l}_j),
   \\ 
   z^3_j=&\tau(\phi_h^{n-1},l_j)+2(\varphi_h^{n-1},l_j),
\\
     z^4_j=&\tau((U_h^n)^2+(u_h^{n-1})^2,l_j)+2(\phi_h^{n-1},l_j)+\tau(\varphi_h^{n-1},l_j).
     \end{aligned}
     \]

Rearranging Equation (\ref{eq1.5}) yields  
\begin{equation}
	\label{eq1.6}
	\left\{\begin{aligned}
&(\mathrm{a})\quad	\hat{\boldsymbol{U}}=\frac{\tau}{2} \hat{\boldsymbol{P}}+\frac{1}{2}\hat{A}^{-1}\boldsymbol{Z}^1, \\
&(\mathrm{b})\quad (2\hat{A}+\frac{\tau^2}{2}\hat{A}+\frac{\tau^2}{2} \hat{B})\hat{\boldsymbol{P}}=\boldsymbol{Z}^2-\frac{\tau}{2}\boldsymbol{Z}^1-\frac{\tau}{2}\hat{B}\hat{A}^{-1}\boldsymbol{Z}^1,\\
&(\mathrm{c})\quad \hat{\boldsymbol{X}}=\frac{\tau}{2} A^{-1}B\hat{\boldsymbol{Y}}+\frac{1}{2}A^{-1}\boldsymbol{Z}^3, \\
&(\mathrm{d})\quad (2A-\frac{\tau^2}{2} B)\hat{\boldsymbol{Y}}=\boldsymbol{Z}^4+\frac{\tau}{2}\boldsymbol{Z}^3,
	\end{aligned}
	\right.
\end{equation}
For sufficiently small time step $\tau$, the coefficient matrices $A$ and $B$ satisfy the positive definiteness condition, which guarantees the unique solvability of the numerical solution generated by Algorithm \ref{alg1}.
\end{proof}

\section{Superconvergence analysis}
Throughout this section, we suppose that the solution $(u, p, \varphi, \phi)$ of system (\ref{eq2.1}) satisfies the following the following regularity property :
\[
\begin{aligned}
&\|u\|_{\mathcal{L}^\infty(J;\mathcal{H}^3(\Omega))}+\|u_t\|_{\mathcal{L}^2(J;\mathcal{H}^2(\Omega))}+\|u_{tt}\|_{\mathcal{L}^2(J;\mathcal{H}^2(\Omega))}+\|u_{ttt}\|_{\mathcal{L}^2(J;\mathcal{H}^1(\Omega))}\\
&\qquad+\|p\|_{\mathcal{L}^\infty(J;\mathcal{H}^2(\Omega))}+\|p_t\|_{\mathcal{L}^2(J;\mathcal{H}^2(\Omega))}
+\|p_{tt}\|_{\mathcal{L}^2(J;\mathcal{H}^1(\Omega))}+\|p_{ttt}\|_{\mathcal{L}^2(J;\mathcal{L}^2(\Omega))}\\
&\qquad+\|\varphi\|_{L^\infty(J;H^2(\Omega))}+\|\varphi_t\|_{L^2(J;H^2(\Omega))}+\|\varphi_{tt}\|_{L^2(J;H^1(\Omega))}+\|\varphi_{ttt}\|_{L^2(J;L^2(\Omega))}\\
&\qquad+\|\phi\|_{L^\infty(J;H^3(\Omega))}+\|\phi_{tt}\|_{L^2(J;H^1(\Omega))}+\|\phi_{ttt}\|_{L^2(J;L^2(\Omega))}\leq C.
\end{aligned}
\]

Let $I_h : H^2(\Omega) \to V_0^h$ be the canonical Lagrange interpolation operator, and introduce the Ritz projection $R_h : H_0^1(\Omega) \to V_0^h$ such that
\begin{equation}\label{eq2.2}
(\nabla(u - R_h u), \nabla v_h) = 0, \quad \forall v_h \in V_0^h.
\end{equation}
It is well-known that $R_h$ admits the following a priori error estimate (see \cite{thomee2006galerkin} )
\begin{equation}
\|u - R_h u\|_0 \leq C h^r \|u\|_k, \quad k = 1, 2, \quad \forall u \in H_0^1(\Omega) \cap H^2(\Omega), 
\end{equation}
and
\begin{equation}
\|\nabla R_h u\|_0 \leq C \|\nabla u\|_0. 
\end{equation}
For simplification of analysis,  the following initial approximate value $(u_h^0,p_h^0,\varphi_h^0,\phi_h^0)=(R_hu_0,  R_hp_0, R_h\varphi_0, R_h\phi_0)$ is considered.
In addition, to further exploit the superconvergence properties, the following estimate in \cite{dongyang2015new} will be used:
\begin{equation}\label{eq2.5}
\|I_h u - R_h u\|_1 \leq C h^2 \|u\|_3,\quad \forall u\in H^3(\Omega).
\end{equation}

To establish the convergence of the fully discrete system (\ref{eq1.3}), the following inductive hypothesis is introduced.

\begin{assumption}\label{assumption3.1}
There exists a mesh parameter $h_0$ $(0 < h_0 < 1)$, such that, for $0 < h \leq h_0$, there holds
\begin{equation}
\|u^m-u_h^m\|_{0,\infty}<1,\quad 0\leq m\leq N.
\end{equation}
\end{assumption}

At time $t=t_{n-\frac{1}{2}}$, it follows from (\ref{eq1.2}) that
\begin{equation}\label{eq1.7}
\left\{\begin{aligned}
&(\mathrm{a})\quad	(\overline{p^n},v_{1h})=(\partial_tu^n,v_{1h})+(u_t^{n-\frac{1}{2}}-\partial_tu^n,v_{1h})+(\overline{p^n}-p^{n-\frac{1}{2}},v_{1h}), &\forall v_{1h}\in\mathcal{V}_0^h, 
\\&(\mathrm{b})\quad	(\nabla\overline{p^n},\nabla v_{1h})=(\nabla\partial_tu^n,\nabla v_{1h})+(\nabla (u_t^{n-\frac{1}{2}}-\partial_tu^n),\nabla v_{1h})+(\nabla(\overline{p^n}-p^{n-\frac{1}{2}}),\nabla v_{1h}), &\forall v_{1h}\in\mathcal{V}_0^h, \\
&(\mathrm{c})\quad	(\partial_tp^n,v_{2h})+(\nabla\overline{u^n},\nabla v_{2h})+(\overline{u^n},v_{2h})+(\overline{\varphi^n}\,\overline{u^n},v_{2h})+(\frac{1}{2}((u^n)^2+(u^{n-1})^2)\overline{u^n},v_{2h})\\
&\qquad=(\partial_tp^n-p_t^{n-\frac{1}{2}},v_{2h})+(\nabla(\overline{u^n}-u^{n-\frac{1}{2}}),\nabla v_{2h})+(\overline{u^n}-u^{n-\frac{1}{2}},v_{2h})
\\
&\qquad\quad+(\overline{\varphi^n}\,\overline{u^n}-\varphi^{n-\frac{1}{2}}\,u^{n-\frac{1}{2}},v_{2h})+(\frac{1}{2}(|u^n|^2+|u^{n-1}|^2)\overline{u^n}-(u^{n-\frac{1}{2}})^3,v_{2h}), &\forall v_{2h}\in\mathcal{V}_0^h,\\
&(\mathrm{d})\quad	(\nabla\overline{\phi^n},\nabla w_{1h})+(\partial_t\varphi^n,w_{1h})=(\nabla(\overline{\phi^n}-\phi^{n-\frac{1}{2}}),\nabla w_{1h})+(\partial_t\varphi^n-\varphi_t^{n-\frac{1}{2}},w_{1h}), &\forall w_{1h}\in V_0^h,\\
&(\mathrm{e})\quad	(\nabla\partial_t\phi^n,\nabla w_{2h})-(\nabla\overline{\varphi^n},\nabla w_{2h})-\frac{1}{2}(\nabla(|u^n|^2+|u^{n-1}|^2),\nabla w_{2h})\\
&\qquad=(\nabla(\partial_t\phi^n-\phi_t^{n-\frac{1}{2}}),\nabla w_{2h})+(\nabla(\varphi^{n-\frac{1}{2}}-\overline{\varphi^n}),\nabla w_{2h})\\
&\qquad\quad+(\nabla(|u^{n-\frac{1}{2}}|^2-\frac{1}{2}(|u^n|^2+|u^{n-1}|^2)),\nabla w_{2h}),&\forall w_{2h}\in V_0^h.
	\end{aligned}
	\right.
\end{equation}

Set 
\[
u^n-u_h^n=u^n-R_hu^n+R_hu^n-u_h^n\overset{\triangle}{=}\eta_u^n-\theta_u^n,
\]
\[
p^n-p_h^n=p^n-R_hp^n+R_hp^n-p_h^n\overset{\triangle}{=}\eta_p^n-\theta_p^n,
\]
\[
\phi^n-\phi_h^n=\phi^n-R_h\phi^n+R_h\phi^n-\phi_h^n\overset{\triangle}{=}\eta_\phi^n-\theta_\phi^n,
\]
\[
\varphi^n-\varphi_h^n=\varphi^n-R_h\varphi^n+R_h\varphi^n-\varphi_h^n\overset{\triangle}{=}\eta_\varphi^n-\theta_\varphi^n.
\]

From (\ref{eq1.3}) to (\ref{eq1.7}), we have the following error residual equations
\begin{equation}\label{eq1.8}
\left\{\begin{aligned}
&(\mathrm{a})\quad		(\overline{\theta_p^n},v_{1h})=-(\overline{\eta_p^n},v_{1h})+(\partial_t\eta_u^n,v_{1h})+(\partial_t\theta_u^n,v_{1h})
\\
&\qquad\quad+(u_t^{n-\frac{1}{2}}-\partial_tu^n,v_{1h})+(\overline{p^n}-p^{n-\frac{1}{2}},v_{1h}), &\forall v_{1h}\in\mathcal{V}_0^h, \\
&(\mathrm{b})\quad	(\nabla\overline{\theta_p^n},\nabla 	v_{1h})=-(\nabla\overline{\eta_p^n},\nabla v_{1h})+(\nabla\partial_t\eta_u^n,\nabla v_{1h})+(\nabla\partial_t\theta_u^n,\nabla v_{1h})
\\
&\qquad\quad+(\nabla(u_t^{n-\frac{1}{2}}-\partial_tu^n),\nabla v_{1h})+(\nabla(\overline{p^n}-p^{n-\frac{1}{2}}),\nabla v_{1h}), &\forall v_{1h}\in\mathcal{V}_0^h, \\
&(\mathrm{c})\quad	(\partial_t\theta_p^n,v_{2h})+(\nabla\overline{\theta_u^n},\nabla v_{2h})+(\overline{\theta_u^n},v_{2h})=-(\partial_t\eta_p^n,v_{2h})-(\nabla\overline{\eta_p^n},\nabla v_{2h})-(\overline{\eta_u^n},v_{2h})
		\\
&\qquad\quad-(\overline{\varphi^n}\,\overline{u^n}-\overline{\varphi_h^n}\,\overline{u_h^n},v_{2h})-\frac{1}{2}((|u^n|^2+|u^{n-1}|^2)\overline{u^n}-(|u_h^n|^2+|u_h^{n-1}|^2)\overline{u_h^n},v_{2h})
\\
&\qquad\quad+(\partial_tp^n-p_t^{n-\frac{1}{2}},v_{2h})+(\nabla(\overline{u^n}-u^{n-\frac{1}{2}}),\nabla v_{2h})+(\overline{u^n}-u^{n-\frac{1}{2}},v_{2h})
\\
&\qquad\quad+(\overline{\varphi^n}\,\overline{u^n}- \varphi^{n-\frac{1}{2}} \, u^{n-\frac{1}{2}} ,v_{2h})+(\frac{1}{2}(|u^n|^2+|u^{n-1}|^2)\overline{u^n}-(u^{n-\frac{1}{2}})^3,v_{2h}), &\forall v_{2h}\in\mathcal{V}_0^h,\\
&(\mathrm{d})\quad	 (\nabla\overline{\theta_\phi^n},\nabla w_{1h})+(\partial_t\theta_\varphi^n,w_{1h})=-(\nabla\overline{\eta_\phi^n},\nabla w_{1h})-(\partial_t\eta_\varphi^n,w_{1h}^n)\\
&\qquad\quad+(\nabla(\overline{\phi^n}-\phi^{n-\frac{1}{2}}),\nabla w_{1h})+(\partial_t\varphi^n-\varphi_t^{n-\frac{1}{2}},w_{1h}), &\forall w_{1h}\in V_0^h,\\
&(\mathrm{e})\quad	(\nabla\partial_t\theta_\phi^n,\nabla w_{2h})=-(\nabla\partial_t\eta_\phi^n,\nabla w_{2h}^n)+(\nabla\overline{\theta_\varphi^n},\nabla w_{2h})+(\nabla\overline{\eta_\varphi^n},\nabla w_{2h})\\
&\qquad\quad-\frac{1}{2}(\nabla(|u^n|^2+|u^{n-1}|^2-|u_h^n|^2-|u_h^{n-1}|^2),\nabla w_{2h})\\
&\qquad\quad+(\nabla(\partial_t\phi^n-\phi_t^{n-\frac{1}{2}}),\nabla w_{2h})
+(\nabla(\varphi^{n-\frac{1}{2}}-\overline{\varphi^n}),\nabla w_{2h})
\\
&\qquad\quad+(\nabla(|u^{n-\frac{1}{2}}|^2-\frac{1}{2}(|u^n|^2+|u^{n-1}|^2)),\nabla w_{2h}),&\forall w_{2h}\in V_0^h.
	\end{aligned}
	\right.
\end{equation}

For Algorithm \ref{alg1}, we have the following convergence theorem.
\begin{theorem}\label{theorem3.1}
Let $(u,p,\phi,\varphi)$ and $
(u_h^n,p_h^n,\phi_h^n,\varphi_h^n)$ denote the exact and fully discrete solutions of (\ref{eq1.1}) and (\ref{eq1.3}), respectively.   The following supercloseness estimate holds:
\begin{equation}\label{1.15}
\|I_hp^n-p_h^n\|_0+\|I_hu^n-u_h^n\|_1+\|I_h\phi^n-\phi_h^n\|_1+\|I_h\varphi^n-\varphi_h^n\|_0\leq C(h^2+\tau^2).
\end{equation}
\end{theorem}
\begin{proof}

Let $v_{1h}=\overline{\theta_u^n}$, $v_{2h}=\overline{\theta_p^n}$ in (\ref{eq1.8})(a), (b), (c), we have
\begin{equation}\label{eq1.9}
\begin{aligned}
&(\partial_t\theta_p^n,\overline{\theta_p^n})+(\nabla\partial_t\theta_u^n,\nabla\overline{\theta_u^n})+(\partial_t\theta_u^n,\overline{\theta_u^n})\\
=&(\nabla\overline{\eta_p^n},\nabla\overline{\theta_u^n})-(\nabla\partial_t\eta_u^n,\nabla\overline{\theta_u^n})-(\nabla(u_t^{n-\frac{1}{2}}-\partial_tu^n),\nabla\overline{\theta_u^n})-(\nabla(\overline{p^n}-p^{n-\frac{1}{2}}),\nabla\overline{\theta_u^n})\\
&+(\overline{\eta_p^n},\overline{\theta_u^n})-(\partial_t\eta_u^n,\overline{\theta_u^n})-(u_t^{n-\frac{1}{2}}-\partial_tu^n,\overline{\theta_u^n})-(\overline{p^n}-p^{n-\frac{1}{2}},\overline{\theta_u^n})
-(\partial_t\eta_p^n,\overline{\theta_p^n})
\\
&-(\nabla\overline{\eta_p^n},\nabla\overline{\theta_p^n})-(\overline{\eta_u^n},\overline{\theta_p^n})-(\overline{\varphi^n}\,\overline{u^n}-\overline{\varphi_h^n}\,\overline{u_h^n},\overline{\theta_p^n})-\frac{1}{2}((|u^n|^2+|u^{n-1}|^2)\overline{u^n}-(|u_h^n|^2+|u_h^{n-1}|^2)\overline{u_h^n},\overline{\theta_p^n})
\\
&+(\partial_tp^n-p_t^{n-\frac{1}{2}},\overline{\theta_p^n})+(\nabla(\overline{u^n}-u^{n-\frac{1}{2}}),\nabla\overline{\theta_p^n})+(\overline{u^n}-u^{n-\frac{1}{2}},\overline{\theta_p^n})+(\overline{\varphi^n}\,\overline{u^n}-\varphi^{n-\frac{1}{2}}\,u^{n-\frac{1}{2}},\overline{\theta_p^n})
\\
&+(\frac{1}{2}(|u^n|^2+|u^{n-1}|^2)\overline{u^n}-(u^{n-\frac{1}{2}})^3,\overline{\theta_p^n})\overset{\triangle}{=}\sum_{i=1}^{18}H_i.
\end{aligned}
\end{equation}

We begin by estimating the terms on the right-hand side of (\ref{eq1.9}). For $H_5$ and $H_{11}$, we know that
\begin{equation}\label{eq1.11}
\begin{aligned}
H_5+H_{11}&=(\overline{\eta_p^n},\overline{\theta_u^n})-(\overline{\eta_u^n},\overline{\theta_u^n})\leq Ch^2\|\overline{p^n}\|_2\|\overline{\theta_u^n}\|_0+Ch^2\|\overline{u^n}\|_2\|\overline{\theta_p^n}\|_0\\
&\leq Ch^4(\|p^n\|_2^2+\|p^{n-1}\|_2^2+\|u^n\|_2^2+\|u^{n-1}\|_2^2)+C(\|\theta_u^n\|_0^2+\|\theta_u^{n-1}\|_0^2+\|\theta_p^n\|_0^2+\|\theta_p^{n-1}\|_0^2).
\end{aligned}
\end{equation}
In fact, it is straightforward to verify that
\begin{equation}\label{eq1.19}
\begin{aligned}
H_6+H_9&=-(\partial_t\eta_u^n,\overline{\theta_u^n})-(\partial_t\eta_p^n,\overline{\theta_p^n})\leq\|\partial_t\eta_u^n\|_0\|\overline{\theta_u^n}\|_0+\|\partial_t\eta_p^n\|_0\|\overline{\theta_p^n}\|_0\\
&\leq \frac{1}{\tau}\int_{t_{n-1}}^{t_n}\|\eta_{ut}\|_0ds\|\overline{\theta_u^n}\|_0+\frac{1}{\tau}\int_{t_{n-1}}^{t_n}\|\eta_{pt}\|_0ds\|\overline{\theta_p^n}\|_0\\
&\leq\frac{Ch^2}{\tau}\int_{t_{n-1}}^{t_n}\|u_t\|_2ds\|\overline{\theta_u^n}\|_0+\frac{Ch^2}{\tau}\int_{t_{n-1}}^{t_n}\|p_t\|_2ds\|\overline{\theta_p^n}\|_0\\
&\leq\frac{Ch^4}{\tau}\int_{t_{n-1}}^{t_n}(\|u_t\|_2^2+\|p_t\|_2^2)ds+C(\|\theta_u^n\|_0^2+\|\theta_u^{n-1}\|_0^2+\|\theta_p^n\|_0^2+\|\theta_p^{n-1}\|_0^2).
\end{aligned}
\end{equation}

Applying Taylor expansion with integral remainder, there holds
\begin{equation}\label{eq1.20}
\begin{aligned}
H_3+H_7+H_{14}&=-(\nabla(u_t^{n-\frac{1}{2}}-\partial_tu^n),\nabla\overline{\theta_u^n})-(u_t^{n-\frac{1}{2}}-\partial_tu^n,\overline{\theta_u^n})+(\partial_tp^n-p_t^{n-\frac{1}{2}},\overline{\theta_p^n})\\
&\leq C\tau\int_{t_{n-1}}^{t_n}\|u_{ttt}\|_1ds\|\overline{\theta_u^n}\|_1  +C\tau\int_{t_{n-1}}^{t_n}\|u_{ttt}\|_0ds\|\overline{\theta_u^n}\|_0 +C\tau\int_{t_{n-1}}^{t_n}\|p_{ttt}\|_0ds\|\overline{\theta_p^n}\|_0\\
&\leq C\tau^3\int_{t_{n-1}}^{t_n}(\|u_{ttt}\|_1^2+\|p_{ttt}\|_0^2)ds+C(\|\theta_u^n\|_1^2+\|\theta_u^{n-1}\|_1^2+\|\theta_p^n\|_0^2+\|\theta_p^{n-1}\|_0^2),
\end{aligned}
\end{equation}
and
\begin{equation}\label{eq1.21}
\begin{aligned}
&H_4+H_8+H_{15}+H_{16}\\
=&-(\nabla(\overline{p^n}-p^{n-\frac{1}{2}}),\nabla\overline{\theta_u^n})-(\overline{p^n}-p^{n-\frac{1}{2}},\overline{\theta_u^n})+(\nabla(\overline{u^n}-u^{n-\frac{1}{2}}),\nabla\overline{\theta_p^n})+(\overline{u^n}-u^{n-\frac{1}{2}},\overline{\theta_p^n})\\
\leq &C\tau\int_{t_{n-1}}^{t_n}\|p_{tt}\|_1ds\|\overline{\theta_u^n}\|_1 + C\tau\int_{t_{n-1}}^{t_n}\|p_{tt}\|_0ds\|\overline{\theta_u^n}\|_0 +C\tau\int_{t_{n-1}}^{t_n}\|u_{tt}\|_2ds\|\overline{\theta_p^n}\|_0+C\tau\int_{t_{n-1}}^{t_n}\|u_{tt}\|_0ds\|\overline{\theta_p^n}\|_0\\
\leq & C\tau^3\int_{t_{n-1}}^{t_n}(\|p_{tt}\|_1^2+\|u_{tt}\|_2^2)ds+C(\|\theta_u^n\|_1^2+\|\theta_u^{n-1}\|_1^2+\|\theta_p^n\|_0^2+\|\theta_p^{n-1}\|_0^2).
\end{aligned}
\end{equation}

According to (\ref{eq2.2}) 
\begin{equation}
	H_{1}+H_{2}+H_{10}=0.
\end{equation}
It can be checked that
\begin{equation}\label{eq1.22}
\begin{aligned}
H_{12}&=-(\overline{\varphi^n}\,\overline{u^n}-\overline{\varphi_h^n}\,\overline{u_h^n},\overline{\theta_p^n})
=-(\overline{\varphi^n}\,\overline{u^n}-\overline{\varphi^n}\,\overline{u_h^n}+\overline{\varphi^n}\,\overline{u_h^n}-\overline{\varphi_h^n}\,\overline{u_h^n},\overline{\theta_p^n})
\\&=-(\overline{\varphi^n}(\overline{u^n}-\overline{u_h^n})+(\overline{\varphi^n}-\overline{\varphi_h^n})\overline{u_h^n},\overline{\theta_p^n})\\
&\leq\|\overline{\varphi^n}\|_{0,\infty}\|\overline{\eta_u^n}+\overline{\theta_u^n}\|_0\|\overline{\theta_p^n}\|_0+\|\overline{\eta_\varphi^n}+\overline{\theta_\varphi^n}\|_0\|\overline{u_h^n}\|_{0,\infty}\|\overline{\theta_p^n}\|_0\\
&\leq Ch^2\|\overline{u^n}\|_2\|\overline{\theta_p^n}\|_0+C\|\overline{\theta_u^n}\|_0\|\overline{\theta_p^n}\|_0+Ch^2\|\overline{\varphi^n}\|_2\|\overline{\theta_p^n}\|_0+C\|\overline{\theta_\varphi^n}\|_0\|\overline{\theta_p^n}\|_0.\\
&\leq Ch^4(\|u^n\|_2^2+\|u^{n-1}\|_2^2+\|\varphi^n\|_2^2+\|\varphi^{n-1}\|_2^2)
\\&\quad+C(\|\theta_u^n\|_0^2+\|\theta_u^{n-1}\|_0^2+\|\theta_\varphi^n\|_0^2+\|\theta_\varphi^{n-1}\|_0^2+\|\theta_p^n\|_0^2+\|\theta_p^{n-1}\|_0^2),
\end{aligned}
\end{equation}
and
\begin{equation}\label{eq1.23}
\begin{aligned}
H_{17}&=(\overline{\varphi^n}\,\overline{u^n}-\varphi^{n-\frac{1}{2}}\,u^{n-\frac{1}{2}},\overline{\theta_p^n})=(\overline{\varphi^n}\,\overline{u^n}-\overline{\varphi^{n-\frac{1}{2}}}\,\overline{u^n}+\overline{\varphi^{n-\frac{1}{2}}}\,\overline{u^n}-\overline{\varphi^{n-\frac{1}{2}}}\,\overline{u^{n-\frac{1}{2}}},\overline{\theta_p^n})\\
&=((\overline{\varphi^n}-\overline{\varphi^{n-\frac{1}{2}}})\overline{u^n}+\overline{\varphi^{n-\frac{1}{2}}}(\overline{u^n}-\overline{u^{n-\frac{1}{2}}}),\overline{\theta_p^n})\\
&\leq \|\overline{\varphi^n}-\varphi^{n-\frac{1}{2}}\|_0 \|\overline{u^n}\|_{0,\infty} \|\overline{\theta_p^n}\|_0+\|\varphi^{n-\frac{1}{2}}\|_{0,\infty}\|\overline{u^n}-u^{n-\frac{1}{2}}\|_0\|\overline{\theta_p^n}\|_0\\
&\leq C\tau^3\int_{t_{n-1}}^{t_n}(\|\varphi_{tt}\|_0^2ds+\|u_{tt}\|_0^2)ds+C(\|\theta_p^n\|_0^2+\|\theta_p^{n-1}\|_0^2).
\end{aligned}
\end{equation}

Similar to \cite{shi2025new}, the following estimates are true:
\begin{equation}\label{eq1.24}
\begin{aligned}
H_{13}&=-\frac{1}{2}((|u^n|^2+|u^{n-1}|^2)\overline{u^n}-(|u_h^n|^2+|u_h^{n-1}|^2)\overline{u_h^n},\overline{\theta_p^n})\\
&\leq Ch^4(\|u^n\|_2^2+\|u^{n-1}\|_2^2)+C(\|\theta_p^n\|_0^2+\|\theta_p^{n-1}\|_0^2),
\end{aligned}
\end{equation}
and
\begin{equation}\label{eq1.12}
\begin{aligned}
H_{18}&=(\frac{1}{2}(|u^n|^2+|u^{n-1}|^2)\overline{u^n}-(u^{n-\frac{1}{2}})^3,\overline{\theta_p^n})\\
&\leq C\tau^3\int_{t_{n-1}}^{t_n}(\|u_t\|_1^4+\|u_{tt}\|_0^2)ds+C(\|\theta_p^n\|_0^2+\|\theta_p^{n-1}\|_0^2).
\end{aligned}
\end{equation}

Substituting the above estimates (\ref{eq1.11})--(\ref{eq1.12}) into (\ref{eq1.9}), we get
\begin{equation}\label{eq1.27}
\begin{aligned}
&\frac{1}{2\tau}(\|\theta_p^n\|_0^2-\|\theta_p^{n-1}\|_0^2+\|\theta_u^n\|_1^2-\|\theta_u^{n-1}\|_1^2+\|\theta_u^n\|_0^2-\|\theta_u^{n-1}\|_0^2)\\
\leq&Ch^4(\|p^n\|_2^2+\|p^{n-1}\|_2^2+\|u^n\|_2^2+\|u^{n-1}\|_2^2+\|\varphi^n\|_2^2+\|\varphi^{n-1}\|_2^2)+Ch^4\tau^{-1}\int_{t_n-1}^{t_n}(\|u_t\|_2^2+\|p_t\|_2^2)ds\\
&+C\tau^3\int_{t_{n-1}}^{t_n}(\|u_{ttt}\|_1^2+\|p_{ttt}\|_0^2+\|p_{tt}\|_1^2+\|u_{tt}\|_2^2+\|\varphi_{tt}\|_0^2+\|u_{t}\|_1^4)ds\\
&+C(\|\theta_u^n\|_1^2+\|\theta_u^{n-1}\|_1^2+\|\theta_\varphi^n\|_0^2+\|\theta_\varphi^{n-1}\|_0^2+\|\theta_p^n\|_0^2+\|\theta_p^{n-1}\|_0^2).
\end{aligned}
\end{equation}
Let $w_{1h}=\overline{\theta_\varphi^n}$ and  $w_{2h}=\overline{\theta_\phi^n}$ in (\ref{eq1.8}), we have
\begin{equation}\label{eq1.28}
\begin{aligned}
	&(\nabla\partial_t\theta_\phi^n,\nabla\overline{\theta_\phi^n})+(\partial_t\theta_\varphi^n,\overline{\theta_\varphi^n}) \\
	=&-(\nabla\overline{\eta_\phi^n},\nabla\overline{\theta_\varphi^n})-(\partial_t\eta_\varphi^n,\overline{\theta_\varphi^n})+(\nabla(\overline{\phi^n}-\phi^{n-\frac{1}{2}}),\nabla\overline{\theta_\varphi^n})
	+(\partial_t\varphi^n-\varphi_t^{n-\frac{1}{2}},\overline{\theta_\varphi^n})
	\\&-(\nabla\partial_t\eta_\phi^n,\nabla\overline{\theta_\phi^n})+(\nabla\overline{\eta_\varphi^n}, \nabla\overline{\theta_\phi^n})-\frac{1}{2}(\nabla(|u^n|^2+|u^{n-1}|^2-|u_h^n|^2-|u_h^{n-1}|^2),\nabla\overline{\theta_\phi^n})
	\\&+(\nabla(\partial_t\phi^n-\phi_t^{n-\frac{1}{2}}),\nabla\overline{\theta_\phi^n})
	+(\nabla(\varphi^{n-\frac{1}{2}}-\overline{\varphi^n}), \nabla\overline{\theta_\phi^n})\\
	&+(\nabla(|u^{n-\frac{1}{2}}|^2-\frac{1}{2}(|u^n|^2+|u^{n-1}|^2)),\nabla\overline{\theta_\phi^n})
	\overset{\triangle}{=}\sum_{i=19}^{28}H_i.
\end{aligned}
\end{equation}
Obviously,
\begin{equation}\label{eq1.29}
H_{19}+H_{23}+H_{24}=0.
\end{equation}

Similar to (\ref{eq1.19}), (\ref{eq1.21}), (\ref{eq1.20}),  we have
\begin{equation}
H_{20}=-(\partial_t\eta_\varphi^n,\overline{\theta_\varphi^n})\leq\frac{Ch^4}{\tau}\int_{t_{n-1}}^{t_n}\|\varphi_t\|^2_2ds+C(\|\theta_\varphi^n\|_0^2+\|\theta_\varphi^{n-1}\|_0^2),
\end{equation}
\begin{equation}
\begin{aligned}
	H_{21}+H_{27}&=(\nabla(\overline{\phi^n}-\phi^{n-\frac{1}{2}}),\nabla\overline{\theta_\varphi^n})+(\nabla(\varphi^{n-\frac{1}{2}}-\overline{\varphi^n}), \nabla\overline{\theta_\phi^n})\\
	&\leq C\tau^3\int_{t_{n-1}}^{t_n}(\|\varphi_{tt}\|_1^2+\|\phi_{tt}\|_2^2)ds+C(\|\theta_\phi^n\|_1^2+\|\theta_\phi^{n-1}\|_1^2+\|\theta_\varphi^n\|_0^2+\|\theta_\varphi^{n-1}\|_0^2),
\end{aligned}
\end{equation}
and
\begin{equation}
\begin{aligned}
	H_{22}+H_{26}&=(\partial_t\varphi^n-\varphi_t^{n-\frac{1}{2}},\overline{\theta_\varphi^n})+(\nabla(\partial_t\phi^n-\phi_t^{n-\frac{1}{2}}),\nabla\overline{\theta_\phi^n})\\
	&\leq C\tau^3\int_{t_{n-1}}^{t_n}(\|\varphi_{ttt}\|_0^2+\|\phi_{ttt}\|_1^2)ds+C(\|\theta_\varphi^n\|_0^2+\|\theta_\varphi^{n-1}\|_0^2+\|\theta_\phi^n\|_1^2+\|\theta_\phi^{n-1}\|_1^2).
\end{aligned}
\end{equation}

To provide more intuitive estimates for $H_{25}$ and $H_{28}$, we denote the real and imaginary parts of $u^n$ and $u_h^n$ by $\alpha^n$, $\beta^n$ and $\alpha_h^n$, $\beta_h^n$, respectively. Thus
\begin{equation}
\begin{aligned}
H_{25}=&-\frac{1}{2}(\nabla(|u^n|^2+|u^{n-1}|^2-|u_h^n|^2-|u_h^{n-1}|^2),\nabla\overline{\theta_\phi^n})
\\=&-(\alpha^n\nabla\alpha^n-\alpha_h^n\nabla\alpha_h^n,\nabla\overline{\theta_\phi^n})-(\alpha^{n-1}\nabla\alpha^{n-1}-\alpha_h^{n-1}\nabla\alpha_h^{n-1},\nabla\overline{\theta_\phi^n})\\
&-(\beta^n\nabla\beta^n-\beta_h^n\nabla\beta_h^n,\nabla\overline{\theta_\phi^n})-(\beta^{n-1}\nabla\beta^{n-1}-\beta_h^{n-1}\nabla\beta_h^{n-1},\nabla\overline{\theta_\phi^n})
\overset{\triangle}{=}\sum_{i=1}^{4}A_i.
\end{aligned}
\end{equation}
It follows that
\begin{equation}
\begin{aligned}
A_1=&(\alpha^n\nabla(R_h\alpha^n-\alpha^n),\nabla\overline{\theta_\phi^n})+((\alpha_h^n-\alpha^n)\nabla\alpha^n,\nabla\overline{\theta_\phi^n})+((\alpha_h^n-R_h\alpha^n)\nabla(R_h\alpha^n-\alpha^n),\nabla\overline{\theta_\phi^n})\\
&+((R_h\alpha^n-\alpha^n)\nabla(R_h\alpha^n-\alpha^n),\nabla\overline{\theta_\phi^n})+(\alpha_h^n\nabla(\alpha_h^n-R_h\alpha^n),\nabla\overline{\theta_\phi^n})\overset{\triangle}{=}\sum_{i=1}^{5}B_i.
\end{aligned}
\end{equation}

Let $\alpha^n-\alpha_h^n=\alpha^n-R_h\alpha^n+R_h\alpha^n-\alpha_h^n\overset{\triangle}{=}\eta_\alpha^n+\theta_\alpha^n$, we get
\begin{equation}
\begin{aligned}
B_2+B_5&\leq(\|\eta_\alpha^n+\theta_\alpha^n\|_0\|\nabla\alpha^n\|_{0,\infty}+\|\alpha_h^n\|_{0,\infty}\|\nabla\theta_\alpha^n\|_0)\|\nabla\overline{\theta_\phi^n}\|_0\\
&\leq Ch^2\|\alpha^n\|_2\|\nabla\overline{\theta_\phi^n}\|+C\|\theta_\alpha^n\|_0\|\nabla\overline{\theta_\phi^n}\|_0+C\|\nabla\theta_\alpha^n\|_0\|\nabla\overline{\theta_\phi^n}\|_0\\
&\leq Ch^4\|\alpha^n\|_2^2+C\|\nabla\theta_\alpha^n\|_0^2+C(\|\theta_\phi^n\|_1^2+\|\theta_\phi^{n-1}\|_1^2),
\end{aligned}
\end{equation}
\begin{equation}
B_3\leq\|\theta_\alpha^n\|_{0,4}\|\nabla\eta_\alpha^n\|_{0,4}\|\nabla\overline{\theta_\phi^n}\|_0\leq C\|\theta_\alpha^n\|_0\|\alpha^n\|_3\|\nabla\overline{\theta_\phi^n}\|_0\leq C\|\theta_\alpha^n\|_0^2+C(\|\theta_\phi^n\|_1^2+\|\theta_\phi^{n-1}\|_1^2),
\end{equation}
and
\begin{equation}
B_4\leq\|\eta_\alpha^n\|_{0,4}\|\nabla\eta_\alpha^n\|_{0,4}\|\nabla\overline{\theta_\phi^n}\|_0\leq Ch^2\|\alpha^n\|_3\|\nabla\overline{\theta_\phi^n}\|_0\leq Ch^4\|\alpha^n\|_3^2+C(\|\theta_\phi^n\|_1^2+\|\theta_\phi^{n-1}\|_1^2),
\end{equation}
where the result follows from the Sobolev embedding theorem $H^3(\Omega) \hookrightarrow W^{2,4}(\Omega)$.

Defining the average of a function $\alpha \in W^{1,\infty}(K)$ as $\hat{\alpha} = \frac{1}{|K|} \int_{K} \alpha d\bm{x}$, we have the approximate property  (\cite{chen2013two}),
\[
\|\alpha - \hat{\alpha}\|_{0,\infty,K} \leq C h \|\alpha\|_{1,\infty,K}.
\]
Applying (\ref{eq2.2}), then yields
\begin{equation}\begin{aligned}
B_1=&((\alpha^n-\hat{\alpha}^n)\nabla\eta_\alpha^n,\nabla\overline{\theta_\phi^n})
\leq Ch^2\sum_{K\in T_h}\|\alpha^n\|_{1,\infty,K}\|\alpha^n\|_{2,K}\|\nabla\overline{\theta_\phi^n}\|_{0,K}
\\\leq &Ch^4\|\alpha^n\|_3+C(\|\theta_\phi^n\|_1^2+\|\theta_\phi^{n-1}\|_1^2).
\end{aligned}\end{equation}
Based on the above estimates, we have
\begin{equation}
A_1\leq Ch^4\|u^n\|_3+C\|\theta_u^n\|_1+C(\|\theta_\phi^n\|_1^2+\|\theta_\phi^{n-1}\|_1^2).
\end{equation}
Similarly, the boundedness of $A_2$, $A_3$, $A_4$ can be established. 

Note that
\begin{equation}
H_{25}\leq Ch^4\|u^n\|_3+C\|\theta_u^n\|_1+C(\|\theta_\phi^n\|_1^2+\|\theta_\phi^{n-1}\|_1^2).
\end{equation}
For $H_{28}$, we have the estimate
\begin{eqnarray}\label{eq1.41}
H_{28}&=&(\nabla(|u^{n-\frac{1}{2}}|^2-\frac{1}{2}(|u^n|^2+|u^{n-1}|^2)),\nabla\overline{\theta_\phi^n})\nonumber
\\
&=&2(\alpha^{n-\frac{1}{2}}\nabla(\alpha^{n-\frac{1}{2}}-\overline{\alpha^n}),\nabla\overline{\theta_\phi^n})+2((\alpha^{n-\frac{1}{2}}-\overline{\alpha^n})\nabla\overline{\alpha^n},\nabla\overline{\theta_\phi^n})\nonumber
\\&&+\frac{1}{2}(\alpha^n\nabla(\alpha^{n-1}-\alpha^n),\nabla\overline{\theta_\phi^n})+\frac{1}{2}(\alpha^{n-1}\nabla(\alpha^n-\alpha^{n-1}),\nabla\overline{\theta_\phi^n})\nonumber
\\
&&+2(\beta^{n-\frac{1}{2}}\nabla(\beta^{n-\frac{1}{2}}-\overline{\beta^n}),\nabla\overline{\theta_\phi^n})+2((\beta^{n-\frac{1}{2}}-\overline{\beta^n})\nabla\overline{\beta^n},\nabla\overline{\theta_\phi^n})\nonumber
\\&&+\frac{1}{2}(\beta^n\nabla(\beta^{n-1}-\beta^n),\nabla\overline{\theta_\phi^n})+\frac{1}{2}(\beta^{n-1}\nabla(\beta^n-\beta^{n-1}),\nabla\overline{\theta_\phi^n})\nonumber
\\
&\leq&C\|\alpha^{n-\frac{1}{2}}\|_{0,\infty}\|\alpha^{n-\frac{1}{2}}-\overline{\alpha^n}\|_1\|\overline{\theta_\phi^n}\|_1\nonumber
+C\|\alpha^{n-\frac{1}{2}}-\overline{\alpha^n}\|_0\|\overline{\alpha^n}\|_{1,\infty}\|\overline{\theta_\phi^n}\|_1\nonumber
\\&&+C\|\alpha^n\|_{0,\infty}\|\alpha^{n-1}-\alpha^n\|_1\|\overline{\theta_\phi^n}\|_1
+C\|\alpha^{n-1}\|_{0,\infty}\|\alpha^{n}-\alpha^{n-1}\|_1\|\overline{\theta_\phi^n}\|_1\nonumber
\\&&+C\|\beta^{n-\frac{1}{2}}\|_{0,\infty}\|\beta^{n-\frac{1}{2}}-\overline{\beta^n}\|_1\|\overline{\theta_\phi^n}\|_1+C\|\beta^{n-\frac{1}{2}}-\overline{\beta^n}\|_0\|\overline{\beta^n}\|_{1,\infty}\|\overline{\theta_\phi^n}\|_1\nonumber
\\
&&+C\|\beta^n\|_{0,\infty}\|\beta^{n-1}-\beta^n\|_1\|\overline{\theta_\phi^n}\|_1+C\|\beta^{n-1}\|_{0,\infty}\|\beta^{n}-\beta^{n-1}\|_1\|\overline{\theta_\phi^n}\|_1\nonumber\\
&\leq&C\tau\int_{t_{n-1}}^{t_{n}}\|\alpha_{tt}\|_1ds\|\overline{\theta_\phi^n}\|_1+C\tau\int_{t_{n-1}}^{t_{n}}\|\alpha_{tt}\|_0ds\|\overline{\theta_\phi^n}\|_1+C\tau\int_{t_{n-1}}^{t_{n}}\|\alpha_{t}\|_1ds\|\overline{\theta_\phi^n}\|_1\nonumber
\\&&+C\tau\int_{t_{n-1}}^{t_{n}}\|\alpha_{t}\|_1ds\|\overline{\theta_\phi^n}\|_1
+C\tau\int_{t_{n-1}}^{t_{n}}\|\beta_{tt}\|_1ds\|\overline{\theta_\phi^n}\|_1+C\tau\int_{t_{n-1}}^{t_{n}}\|\beta_{tt}\|_0ds\|\overline{\theta_\phi^n}\|_1\nonumber
\\&&+C\tau\int_{t_{n-1}}^{t_{n}}\|\beta_{t}\|_1ds\|\overline{\theta_\phi^n}\|_1+C\tau\int_{t_{n-1}}^{t_{n}}\|\beta_{t}\|_1ds\|\overline{\theta_\phi^n}\|_1\nonumber
\\
&\leq&C\tau^3\int_{t_{n-1}}^{t_n}(\|\alpha_{tt}\|_1^2+\|\alpha_{t}\|_1^2+\|\beta_{tt}\|_1^2+\|\beta_{t}\|_1^2)ds+C(\|\theta_\phi^n\|_1^2+\|\theta_\phi^{n-1}\|_1^2)\nonumber
\\
&\leq&C\tau^3\int_{t_{n-1}}^{t_n}(\|u_{tt}\|_1^2+\|u_{t}\|_1^2)ds+C(\|\theta_\phi^n\|_1^2+\|\theta_\phi^{n-1}\|_1^2).
\end{eqnarray}
Substituting these results (\ref{eq1.29})--(\ref{eq1.41}) into (\ref{eq1.28}), we get
\begin{equation}\label{eq1.42}
\begin{aligned}
&\frac{1}{2\tau}(\|\theta_\phi^n\|_1^2-\|\theta_\phi^{n-1}\|_1^2+\|\theta_\varphi^n\|_0^2-\|\theta_\varphi^{n-1}\|_0^2)\\
\leq&Ch^4\tau^{-1}\int_{t_{n-1}}^{t_n}\|\varphi_t\|^2_2ds+C\tau^3\int_{t_{n-1}}^{t_n}(\|\varphi_{tt}\|_1^2+\|\phi_{tt}\|_2^2+\|\varphi_{ttt}\|_0^2+\|\phi_{ttt}\|_1^2+\|u_{tt}\|_1^2+\|u_{t}\|_1^2)ds\\
&+Ch^4\|u^n\|_3+C(\|\theta_u^n\|_1+\|\theta_\phi^n\|_1^2+\|\theta_\phi^{n-1}\|_1^2+\|\theta_\varphi^n\|_0^2+\|\theta_\varphi^{n-1}\|_0^2).
\end{aligned}
\end{equation}
Then, combining (\ref{eq1.27}) and (\ref{eq1.42}), we have
\begin{equation}\label{eq1.43}
\begin{aligned}
&\frac{1}{2\tau}(\|\theta_p^n\|_0^2-\|\theta_p^{n-1}\|_0^2+\|\theta_u^n\|_1^2-\|\theta_u^{n-1}\|_1^2+\|\theta_u^n\|_0^2-\|\theta_u^{n-1}\|_0^2+\|\theta_\phi^n\|_1^2-\|\theta_\phi^{n-1}\|_1^2+\|\theta_\varphi^n\|_0^2-\|\theta_\varphi^{n-1}\|_0^2)\\
\leq&Ch^4(\|p^n\|_2^2+\|p^{n-1}\|_2^2+\|u^n\|_3^2+\|u^{n-1}\|_2^2+\|\varphi^n\|_2^2+\|\varphi^{n-1}\|_2^2)+Ch^4\tau^{-1}\int_{t_n-1}^{t_n}(\|u_t\|_2^2+\|p_t\|_2^2+\|\varphi_t\|^2_2)ds\\
&+C\tau^3\int_{t_{n-1}}^{t_n}(\|u_{ttt}\|_1^2+\|p_{ttt}\|_0^2+\|p_{tt}\|_1^2+\|u_{tt}\|_2^2+\|\varphi_{tt}\|_1^2+\|u_{t}\|_1^4+\|\phi_{tt}\|_2^2+\|\varphi_{ttt}\|_0^2+\|\phi_{ttt}\|_1^2+\|u_{t}\|_1^2)ds\\
&+C(\|\theta_u^n\|_1^2+\|\theta_u^{n-1}\|_1^2+\|\theta_\varphi^n\|_0^2+\|\theta_\varphi^{n-1}\|_0^2+\|\theta_p^n\|_0^2+\|\theta_p^{n-1}\|_0^2+\|\theta_\phi^n\|_1^2+\|\theta_\phi^{n-1}\|_1^2).
\end{aligned}
\end{equation}
Summing (\ref{eq1.43}) for $n$ from 1 to $m$ (where $1\leq m\leq N$) leads to
\begin{equation}\label{eq1.44}
\begin{aligned}
&(1-C\tau)(\|\theta_p^m\|_0^2+\|\theta_u^m\|_1^2+\|\theta_\phi^m\|_1^2+\|\theta_\varphi^m\|_0^2)
\\\leq& \|\theta_p^0\|_0^2+\|\theta_u^0\|_1^2+\|\theta_u^0\|_0^2+\|\theta_\phi^0\|_1^2+\|\theta_\varphi^0\|_0^2+Ch^4+C\tau^4\\
&+C\tau\sum_{n=0}^{m-1}(\|\theta_p^n\|_0^2+\|\theta_u^n\|_1^2+\|\theta_\phi^n\|_1^2+\|\theta_\varphi^n\|_0^2).
\end{aligned}
\end{equation}
Noting that $\theta_p^0=\theta_u^0=\theta_\varphi^0=\theta_\phi^0=0$, when $1-C\tau>0$, we apply the discrete Gronwall lemma to get 
\begin{equation}\label{eq3.34}
\|R_hp^n-p_h^n\|_0+\|R_hu^n-u_h^n\|_1+\|R_h\phi^n-\phi_h^n\|_1+\|R_h\varphi^n-\varphi_h^n\|_0\leq C(h^2+\tau^2).
\end{equation}
Finally, using the triangle inequality and (\ref{eq2.5}), we obtain the desire result (\ref{1.15}).

Now we are ready to prove  Assumption \ref{assumption3.1} by mathematical induction. When $m=0$, we get $\|u^0-u_h^0\|_{0,\infty}=\|u^0-R_hu^0\|_{0,\infty}\leq Ch^{2-\frac{d}{2}}<1$ for a small $h$. Assume that $\|u^m-u_h^m\|_{0,\infty}<1$ is true for $m \leq n-1$, we have from
(\ref{eq3.34}) that
\begin{equation}
\|R_hu^{n-1}-u_h^{n-1}\|_{0,\infty}\leq Ch^{-\frac{d}{2}}(h^2+\tau^2).
\end{equation}
For $m = n$, since $\|u-u_h\|_{0,\infty}$ is continuous with respect to $t$, for a sufficiently small $\epsilon$, there exist  $h\leq h_0$, such
that when $|t_n-t_{n-1}|<\epsilon$, there holds
\begin{equation}
\big|\,\|u^{n}-u_h^{n}\|_{0,\infty}-\|u^{n-1}-u_h^{n-1}\|_{0,\infty}\,\big|<h^{2-\frac{d}{2}},
\end{equation}
which implies for $\tau=O(h)$,
\begin{equation}\begin{aligned}
\|u^{n}-u_h^{n}\|_{0,\infty}<&\|R_hu^{n-1}-u_h^{n-1}\|_{0,\infty}+\|R_hu^{n-1}-u^{n-1}\|_{0,\infty}+h^{2-\frac{d}{2}}
\\ \leq& Ch^{-\frac{d}{2}}(h^2+\tau^2)+Ch^{2-\frac{d}{2}}\leq C^*h^{2-\frac{d}{2}}<1.
\end{aligned}\end{equation}
\end{proof}

\begin{remark}\label{rem2}
Applying the interpolated postprocessing operator $I_{2h}$ introduced in \cite{lin1996construction}, we obtain the global superconvergence result:
\begin{equation}
\|I_{2h}u_h^n-u^n\|_1+\|I_{2h}\phi_h^n-\phi^n\|_1\leq C(h^2+\tau^2).
\end{equation}
Combining the triangle inequality with Theorem \ref{theorem3.1} yields the optimal error estimates
\begin{equation}
\|p_h^n-p^n\|_0+\|\varphi_h^n-\varphi^n\|_0\leq C(h^2+\tau^2).
\end{equation}
\end{remark}
	

\begin{remark}
The proposed combined technique offers two key advantages. First, it effectively relaxes the regularity requirements imposed on the exact solutions: specifically, our theoretical analysis only assumes $u_t \in L^2(J; \mathcal{H}^2(\Omega))$, $p \in L^\infty(J; \mathcal{H}^2(\Omega))$, and $\phi \in L^\infty(J; H^2(\Omega))$. In contrast, a conventional error analysis based solely on the standard interpolation operator would require significantly stronger smoothness, namely, $u_t \in L^2(J; \mathcal{H}^3(\Omega))$, $p \in L^\infty(J; \mathcal{H}^3(\Omega))$, and $\phi \in L^\infty(J; H^3(\Omega))$, to establish the convergence results stated in Theorem \ref{theorem3.1}. Second, this hybrid approach facilitates the derivation of superconvergence properties for $u$ and $\phi$. By contrast, relying exclusively on the Ritz projection operator $R_h$ renders the construction of a postprocessing interpolation operator $I_{2h}$ satisfying the critical identity $I_{2h} R_h \mu = I_{2h} \mu$ (for $\mu = u, \phi$) infeasible under current analytical frameworks.
\end{remark}

\section{Numerical results}
In this section, some numerical results are provided to validate our theoretical analysis. Specifically, the simulations serve triple complementary purposes: (i) to verify the convergence rates predicted by the error estimates,  (ii) to confirm the conservation of total energy by the proposed method, and (iii) to simulate the wave propagation and interaction behaviors.

\subsection{Two-dimensional case}
\begin{example}\label{example4.1.1}
(Accurancy test) The analytical solutions are given by
\begin{equation}
\renewcommand{\arraystretch}{1.5}
\left\{\begin{array}{llllll}
u=e^{-2t}\text{sin}(3\pi x)\text{sin}(3\pi y)+ie^{-3t}\text{sin}(2\pi x)\text{sin}(2\pi y),\nonumber\\
\varphi=18\pi^2e^{-t}\text{sin}(3\pi x)\text{sin}(3\pi y).\nonumber
\end{array}
\right.
\end{equation}
\end{example}


Let $\Omega = (0,1) \times (0,1)$ and $T = 1.0$. The spatial domain $\Omega$ is discretized using a uniform rectangular mesh with $M$ subdivisions per direction, so that the spatial mesh size satisfies $h = 1/M$. To rigorously verify the convergence orders predicted by the theoretical analysis, the time step is selected as $\tau = h$, such that the temporal resolution is well matched with the spatial discretization scale. At the final time $T = 1$, the following errors are computed:  
$\|I_h u^n - u_h^n\|_1$, $\|I_{2h} u_h^n - u^n\|_1$, $\|I_h \phi^n - \phi_h^n\|_1$, $\|I_{2h} \phi_h^n - \phi^n\|_1$, $\|p_h^n - p^n\|_0$, and $\|\varphi_h^n - \varphi^n\|_0$  
and tabulated in Table \ref{table4.1}. 
\begin{table}[!htbp]
	\caption{ Errors and rates at $t = 1.0$ for 2D case.}\label{table4.1}
	\centering
	\begin{tabular*}{\textwidth}{@{\extracolsep{\fill}}l|lllll }
		\toprule
		$M \times M$ & $8 \times 8$ & $16 \times 16$ & $32 \times 32$ & $64 \times 64$ & $128 \times 128$ \\
		\midrule
		$\|I_hu^n - u_h^n\|_1$       & 1.1844e+00 & 4.3908e-01 & 1.2131e-01 & 3.1100e-02 & 7.8242e-03 \\
		order                        & -          & 1.4316     & 1.8557     & 1.9637     & 1.9909       \\ \hline
		$\|I_{2h} u_h^n - u^n\|_1$   & 1.2551e+00 & 4.5442e-01 & 1.2425e-01 & 3.1761e-02 & 7.9846e-03 \\
		order                        & -          & 1.4657     & 1.8708     & 1.9679     & 1.9920       \\\hline
		$\|I_h\phi^n - \phi_h^n\|_1$ & 8.6058e-01 & 1.9977e-01 & 4.8354e-02 & 1.1978e-02 & 2.9874e-03 \\
		order                        & -         & 2.1070     & 2.0466     & 2.0133     & 2.0034       \\\hline
		$\|I_{2h}\phi_h^n-\phi^n\|_1$& 8.9303e-01 & 2.0182e-01 & 4.8588e-02 & 1.2019e-02 & 2.9967e-03 \\
		order                        & -        & 2.1457     & 2.0544     & 2.0152     & 2.0039       \\\hline
		$\|p_h^n-p^n\|_0$            & 1.5240e+00 & 4.1258e-01 & 1.0452e-01 & 2.6198e-02 & 6.5534e-03 \\
		order                        & -          & 1.8852     & 1.9809     & 1.9962     & 1.9991       \\\hline
		$\|\varphi_h^n-\varphi^n\|_0$& 1.4777e+00 & 2.1950e-01 & 4.3216e-02 & 1.0029e-02 & 2.4582e-03 \\
		order                        & -         & 2.7510     & 2.3446     & 2.1073     & 2.0286       \\
		\bottomrule
	\end{tabular*}
\end{table}

As shown therein, $\|I_h u^n - u_h^n\|_1$ and $\|I_h \phi^n - \phi_h^n\|_1$ exhibit second-order convergence in $h$, precisely matching the optimal rate established in Theorem \ref{theorem3.1}. Moreover, $\|I_{2h} u_h^n - u^n\|_1$, $\|I_{2h} \phi_h^n - \phi^n\|_1$, $\|p_h^n - p^n\|_0$, and $\|\varphi_h^n - \varphi^n\|_0$ all achieve $O(h^2)$ convergence-fully corroborating the supercloseness and error estimates derived in Remark \ref{rem2}. For visual confirmation, the corresponding convergence curves are plotted in Fig. \ref{fig4.1}.

\begin{figure}[!htbp]
	\centering
	\subfigure[$u_h^n$, $\phi_h^n$.]{\includegraphics[width=0.35\linewidth]{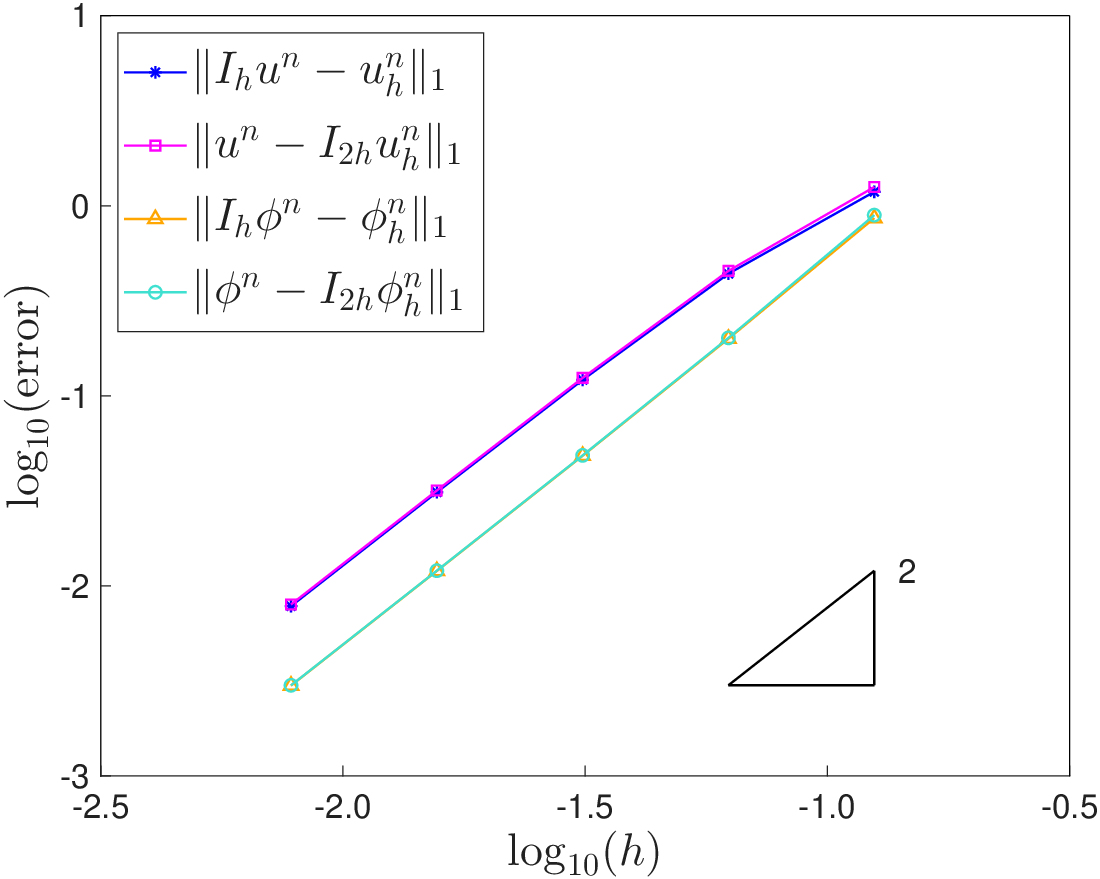}
		}\hspace{1cm}
\subfigure[$p_h^n$, $\varphi_h^n$.]{\includegraphics[width=0.35\linewidth]{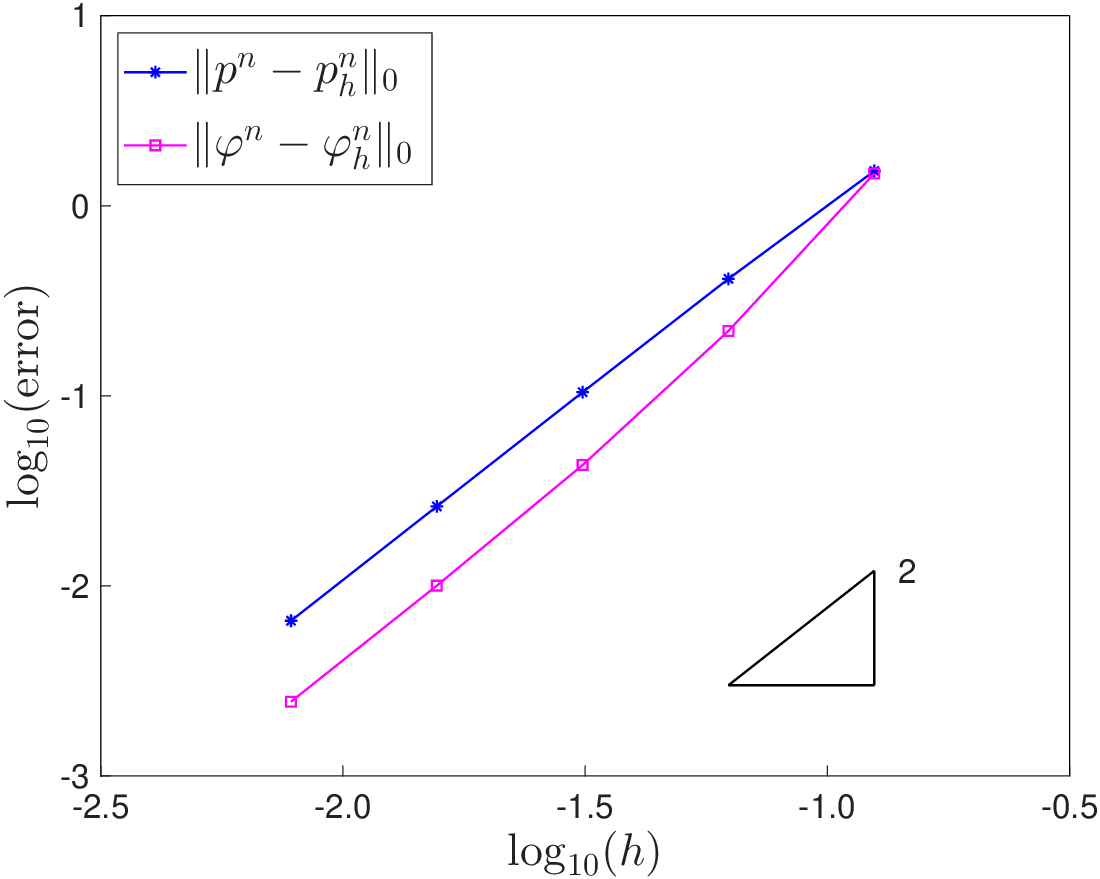}}
\caption{\ Error curves of $u_h^n$, $\phi_h^n$, $p_h^n$, and $\varphi_h^n$ for 2D case.}
	\label{fig4.1}
\end{figure}

\begin{example}\label{example4.1.2}
(Energy conservation $E^n=E^0$) The initial conditions are given by
\begin{equation}
	\renewcommand{\arraystretch}{1.5}
	\left\{\begin{array}{llllll}
		u_0=\text{sin}(3\pi x)\text{sin}(3\pi y)+i\text{sin}(2\pi x)\text{sin}(2\pi y),\\
		\varphi_0=18\pi^2\text{sin}(3\pi x)\text{sin}(3\pi y).\nonumber
	\end{array}
	\right.
\end{equation}
\end{example}
Let $\Omega = (0,1) \times (0,1)$ and $T = 5$. Figure \ref{fig4.2}(a) displays the temporal
evolution of the discrete energy $E^n$, confirming that the proposed scheme exactly preserves the total energy, i.e., $E^n = E^0$ for all $n \geq 0$. To quantify numerical drift more sensitively, Figure \ref{fig4.2}(b) plots the deviation $|E^n - E^0|$ over time; the results show that this quantity almost remains bounded within machine precision, providing strong numerical evidence of exact discrete energy conservation.

\begin{figure}[htbp]
	\centering
	\subfigure[Energy conservation.]{\includegraphics[width=0.35\linewidth]{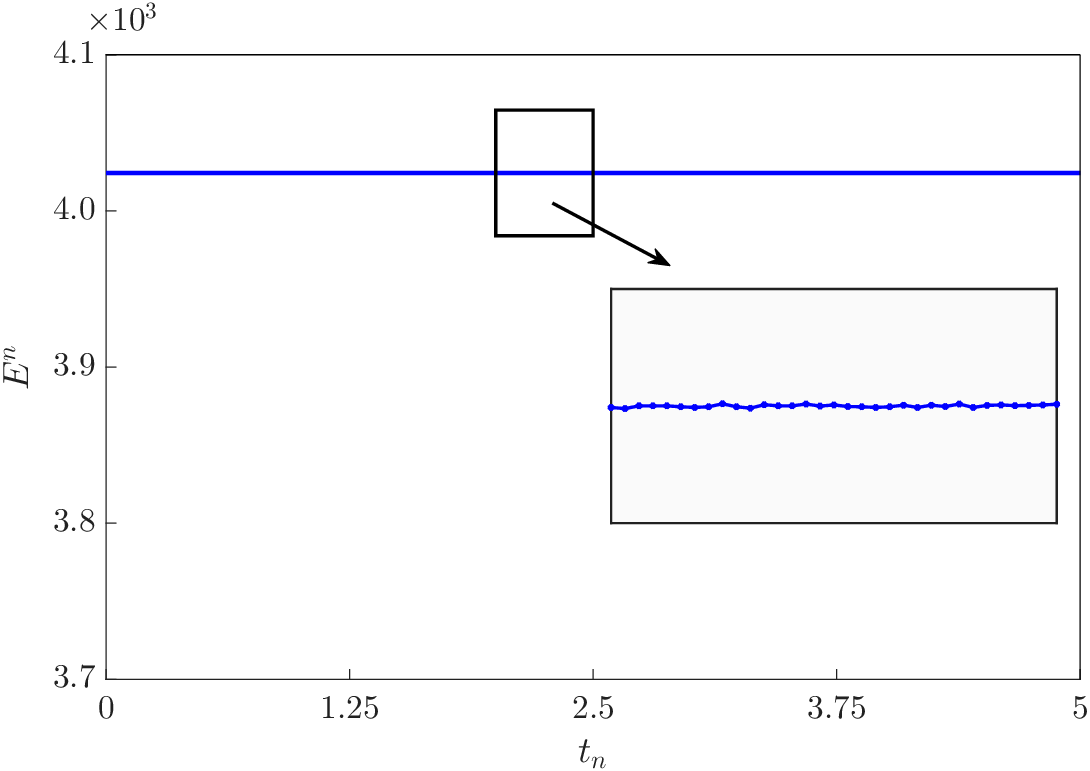}
		}\hspace{1cm}
	\subfigure[Evolution of energy $|E^n-E^0|$.]{\includegraphics[width=0.35\linewidth]{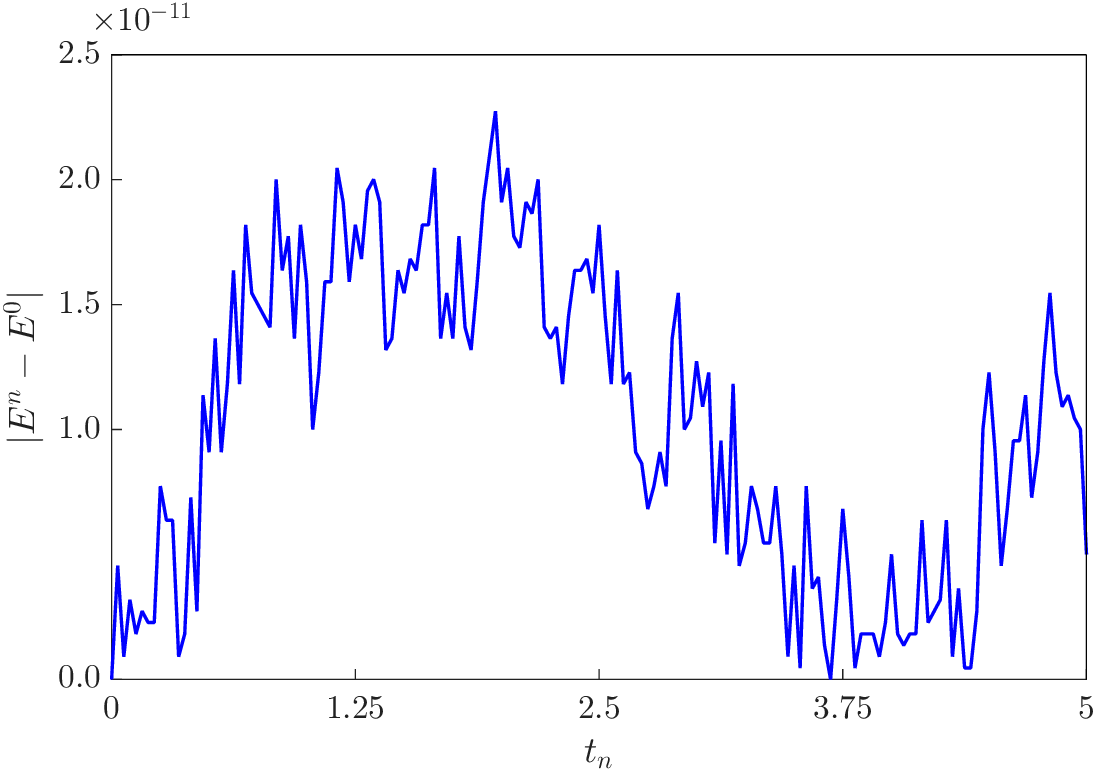}
		}
	\caption{\ Energy curves.}
	\label{fig4.2}
\end{figure}

\begin{example}\label{example4.1.3}
	(Dynamics of wave interactions) The initial conditions are given by (see \cite{bao2016uniformly})
	\begin{equation}
		\renewcommand{\arraystretch}{1.5}
		\left\{\begin{array}{llllll}
			u_0=(1+\frac{i}{2})[\text{exp}(-(x+2)^2-y^2)+\text{exp}(-(x-2)^2-y^2)],\\
			\varphi_0=\text{sech}(x^2+(y+2)^2)+\text{sech}(x^2+(y-2)^2),\nonumber\\ 
			u_1=(1+\frac{i}{2})\text{exp}(-x^2-y^2),\\
			\varphi_1=\text{sech}(x^2+y^2).
		\end{array}
		\right.
	\end{equation}
\end{example}
 
Let $\Omega=(-10,10)\times(-10,10)$, $T=5$, $h=1/8$, $\tau=0.01$. The contour plots illustrating the evolution of $|u|$ and $\varphi$ at distinct time levels are shown in Fig. \ref{fig4.4-1} and Fig. \ref{fig4.4-2}. These visualizations clearly capture the characteristic wave propagation and interaction phenomena governed by the two-dimensional KGZ system (\ref{eq1.1}), demonstrating that the proposed method accurately and efficiently resolves the underlying physical dynamics.
\begin{figure}[!htbp]
	\centering{
	\subfigure[$t=0$]
	{\includegraphics[width=0.25\linewidth]{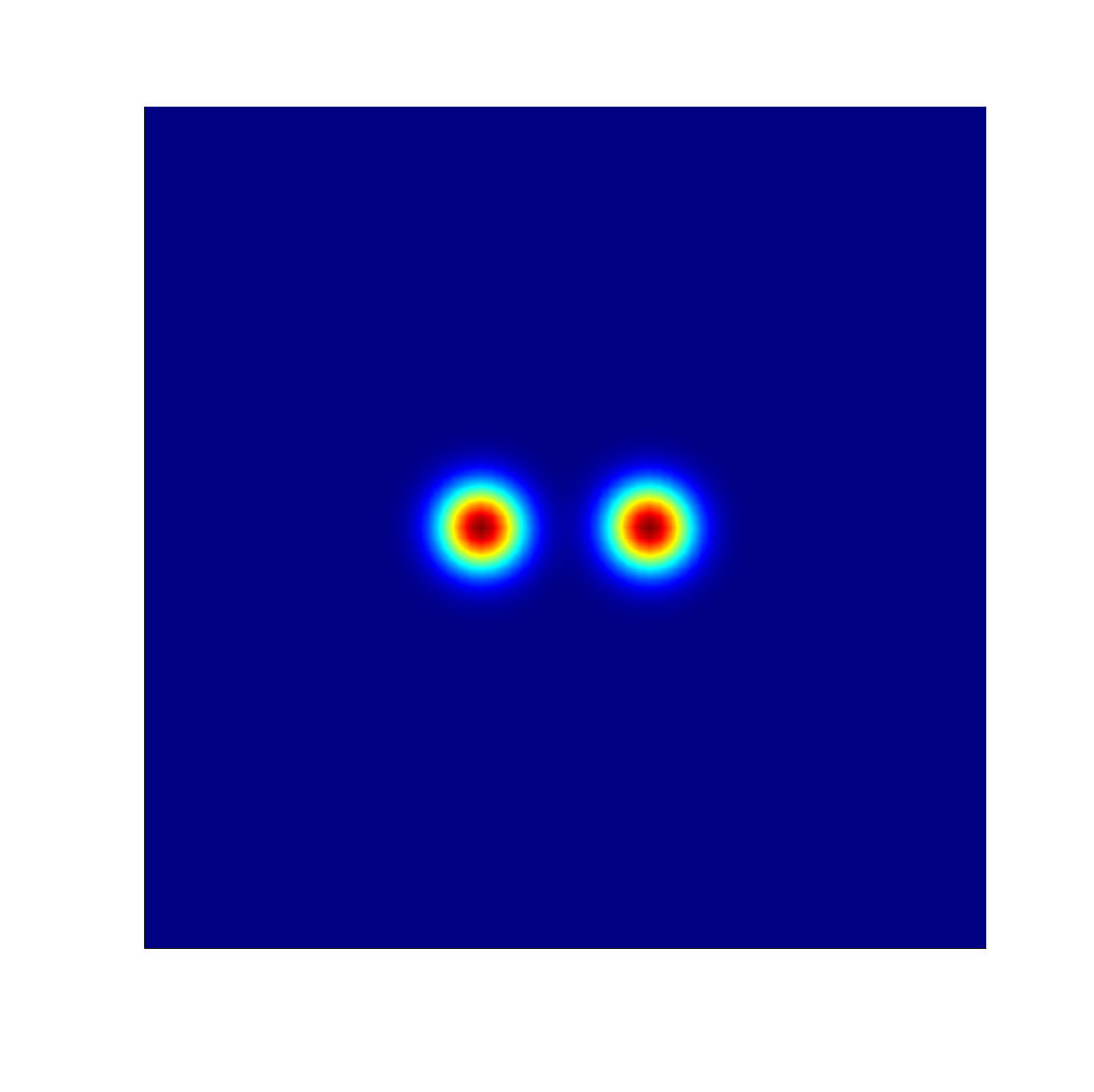}
	}
	\subfigure[$t=1$]
	{
		\includegraphics[width=0.25\linewidth]{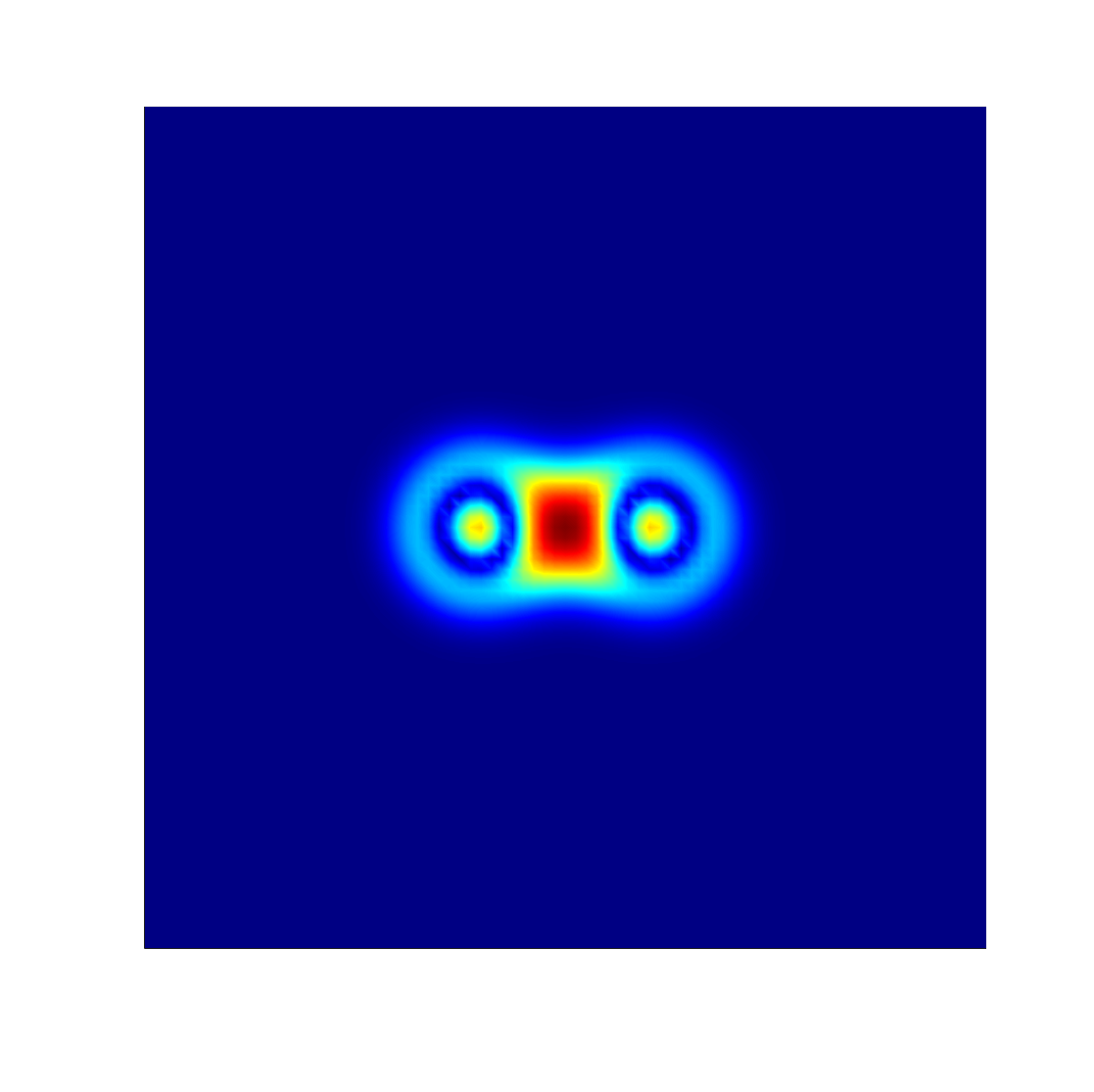}
		}
	\subfigure[$t=2$]
	{\includegraphics[width=0.25\linewidth]{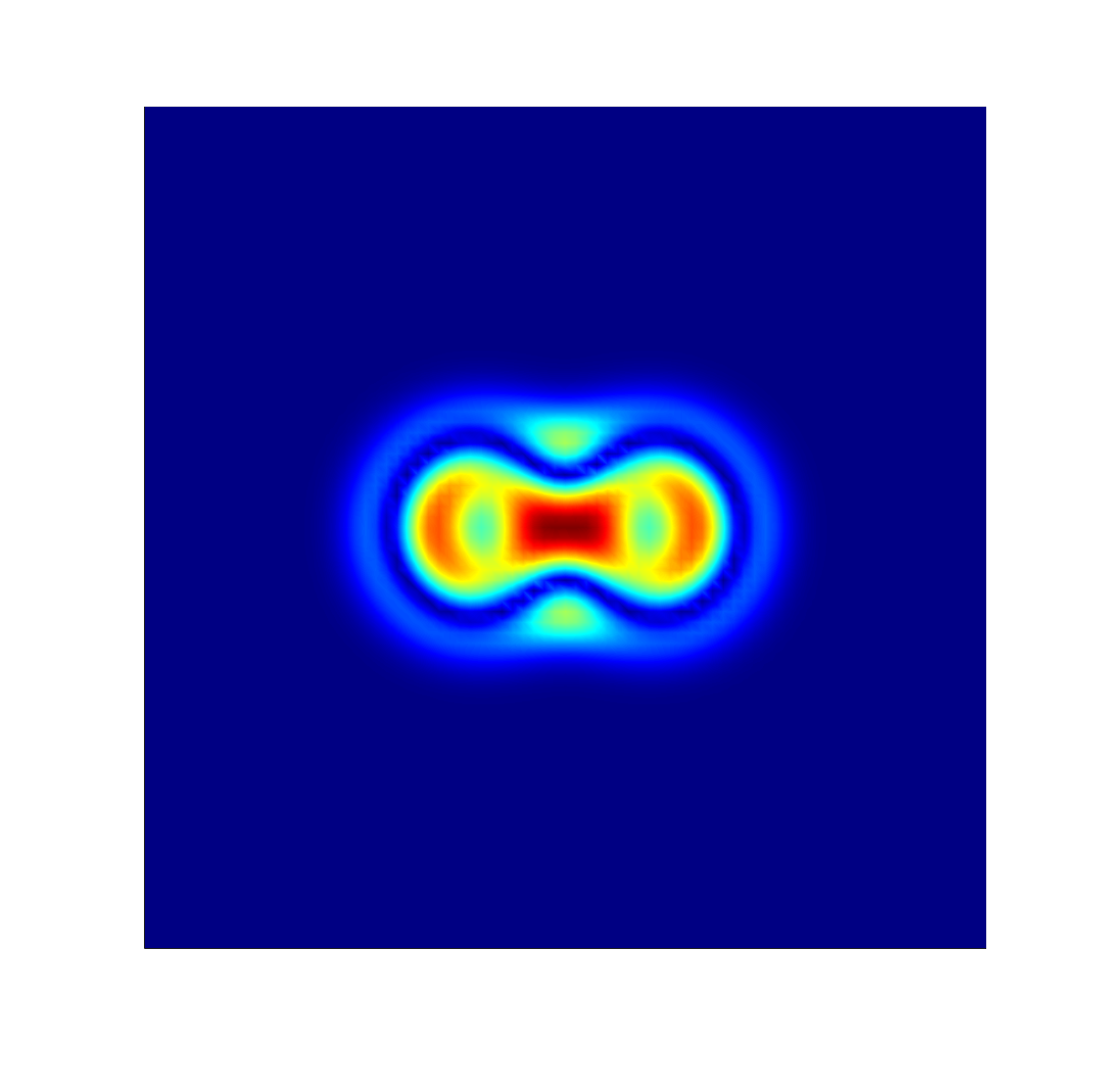}
		}
	\subfigure[$t=3$]
	{
		\includegraphics[width=0.25\linewidth]{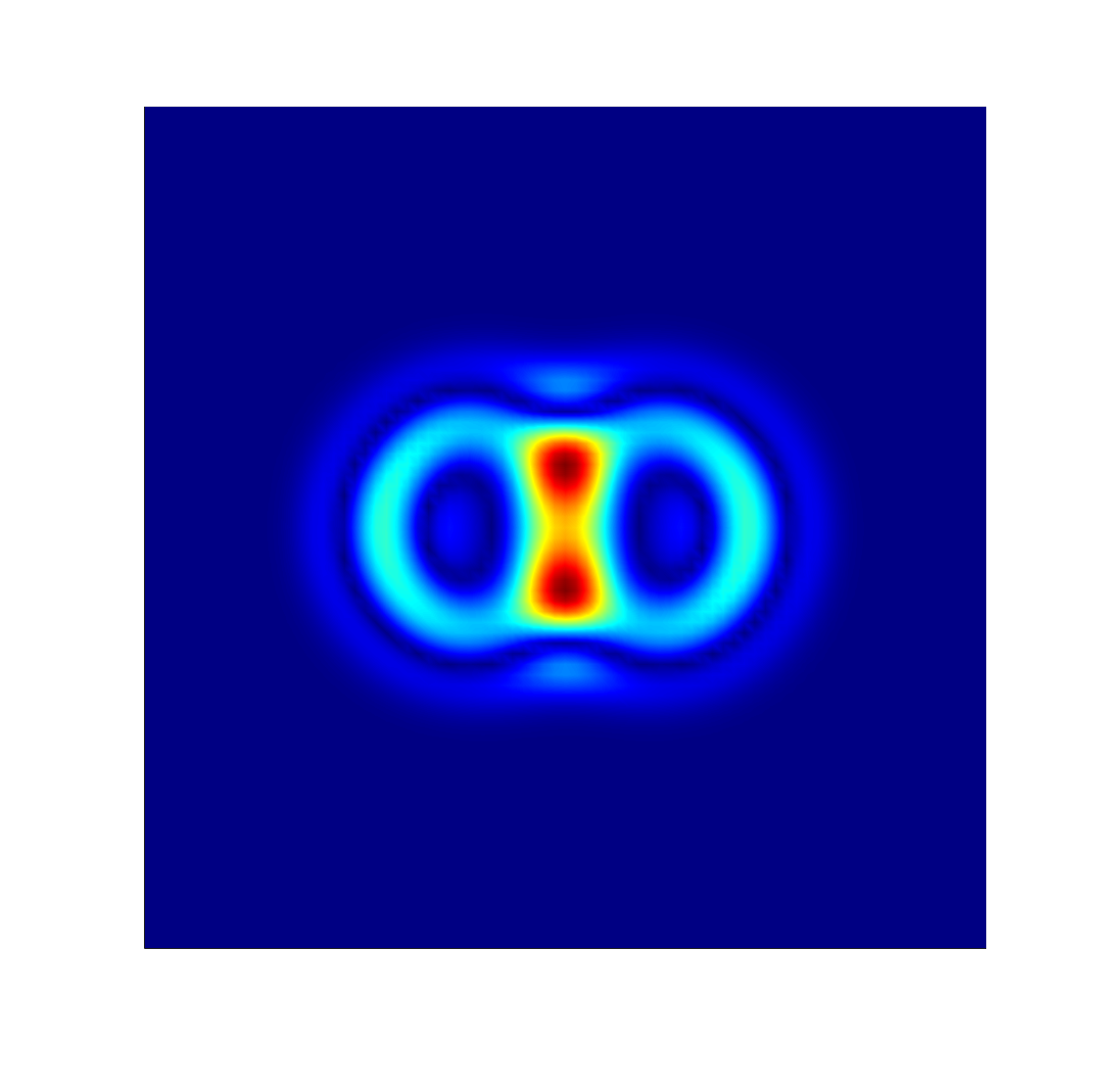}
		}
	\subfigure[$t=4$]
	{
		\includegraphics[width=0.25\linewidth]{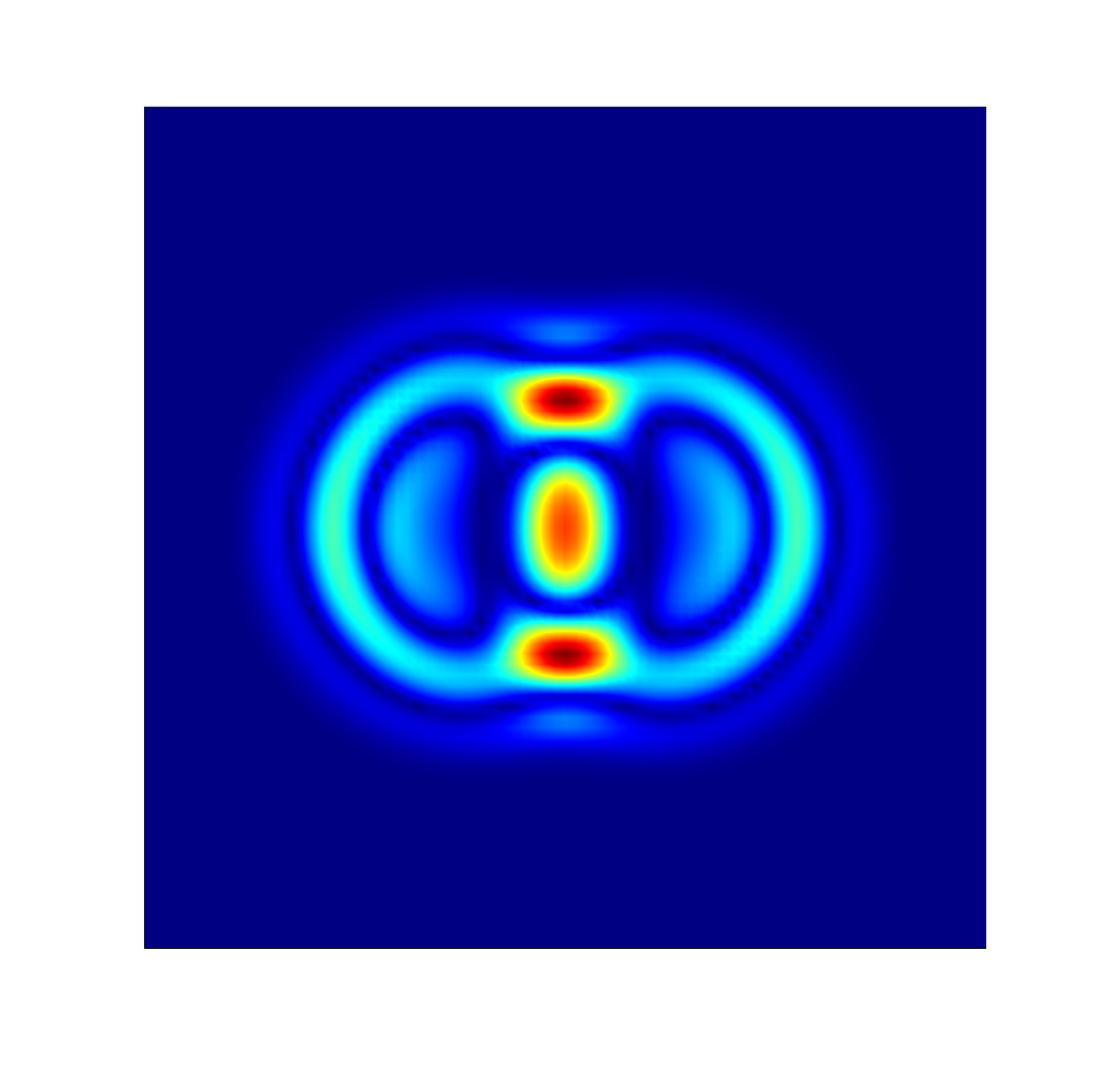}
		}
	\subfigure[$t=5$]
	{
		\includegraphics[width=0.25\linewidth]{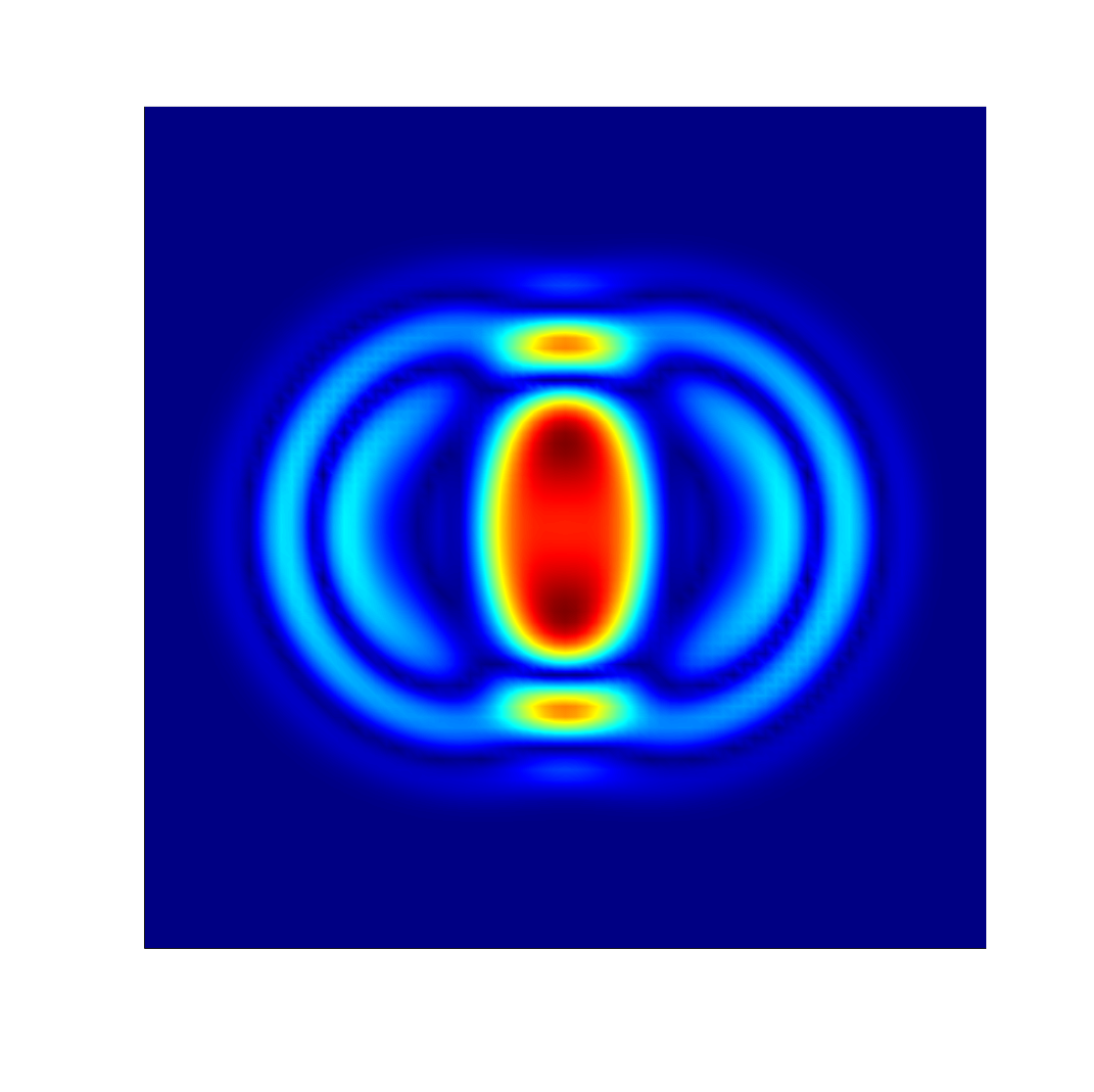}
		}}
	\caption{\ Temporal evolution of the approximate solution $|u|$ depicted via contour lines.}
	\label{fig4.4-1}
\end{figure}

\begin{figure}[!htbp]
	\centering
	{\subfigure[$t=0$]
	{
		\includegraphics[width=0.25\linewidth]{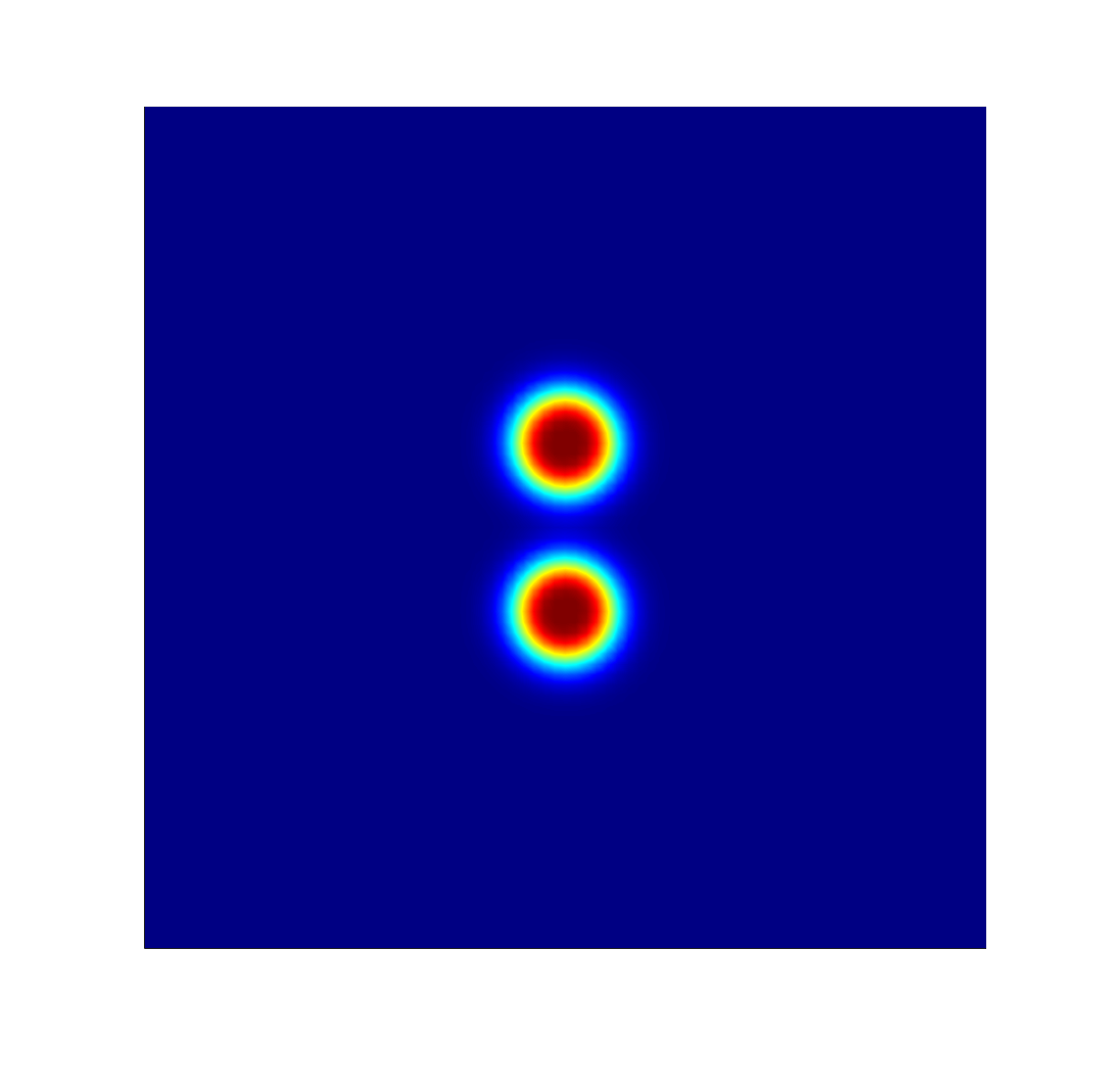}
		}
	\subfigure[$t=1$]
	{		\includegraphics[width=0.25\linewidth]{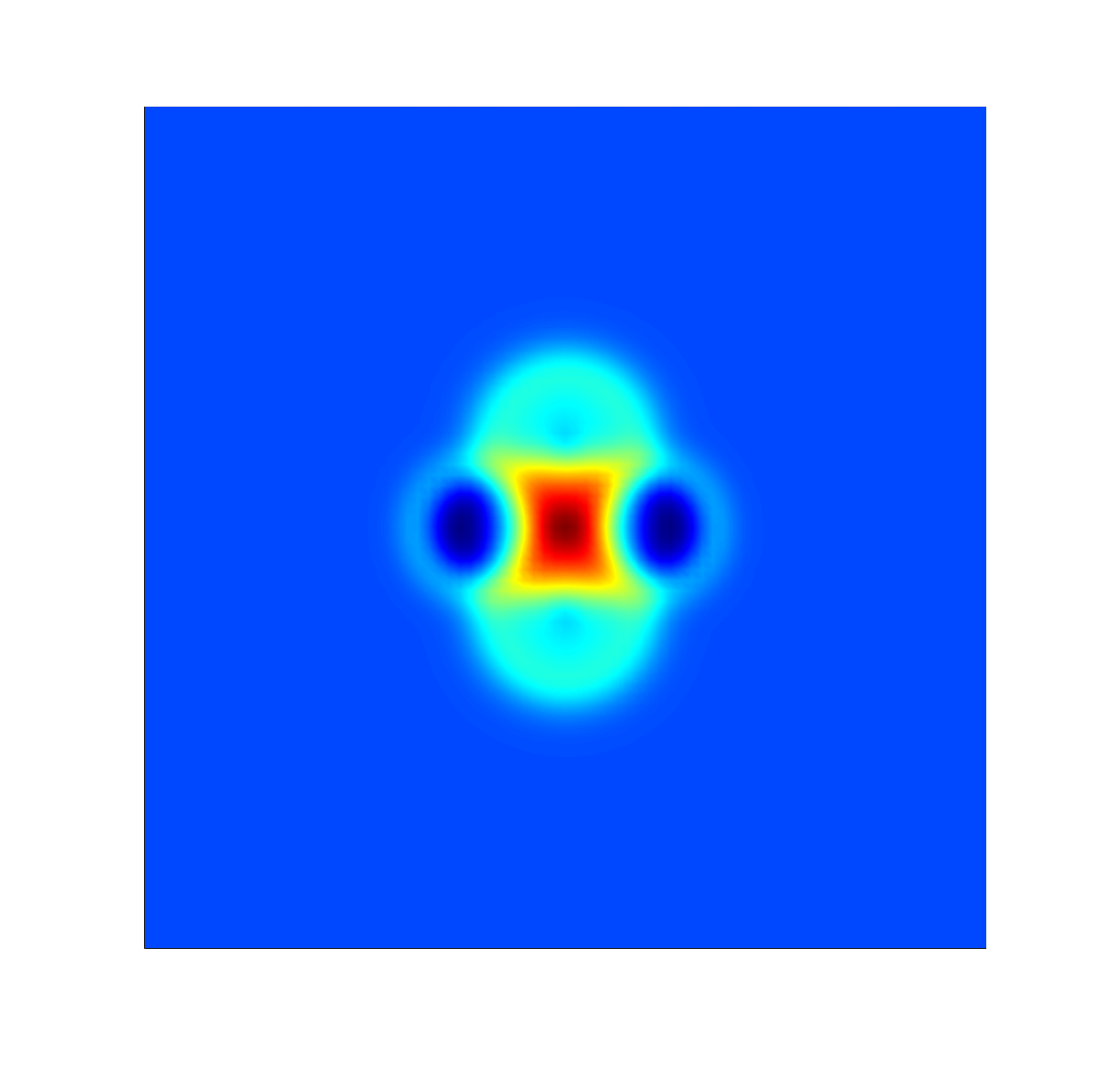}
		}
	\subfigure[$t=2$]
	{
		\includegraphics[width=0.25\linewidth]{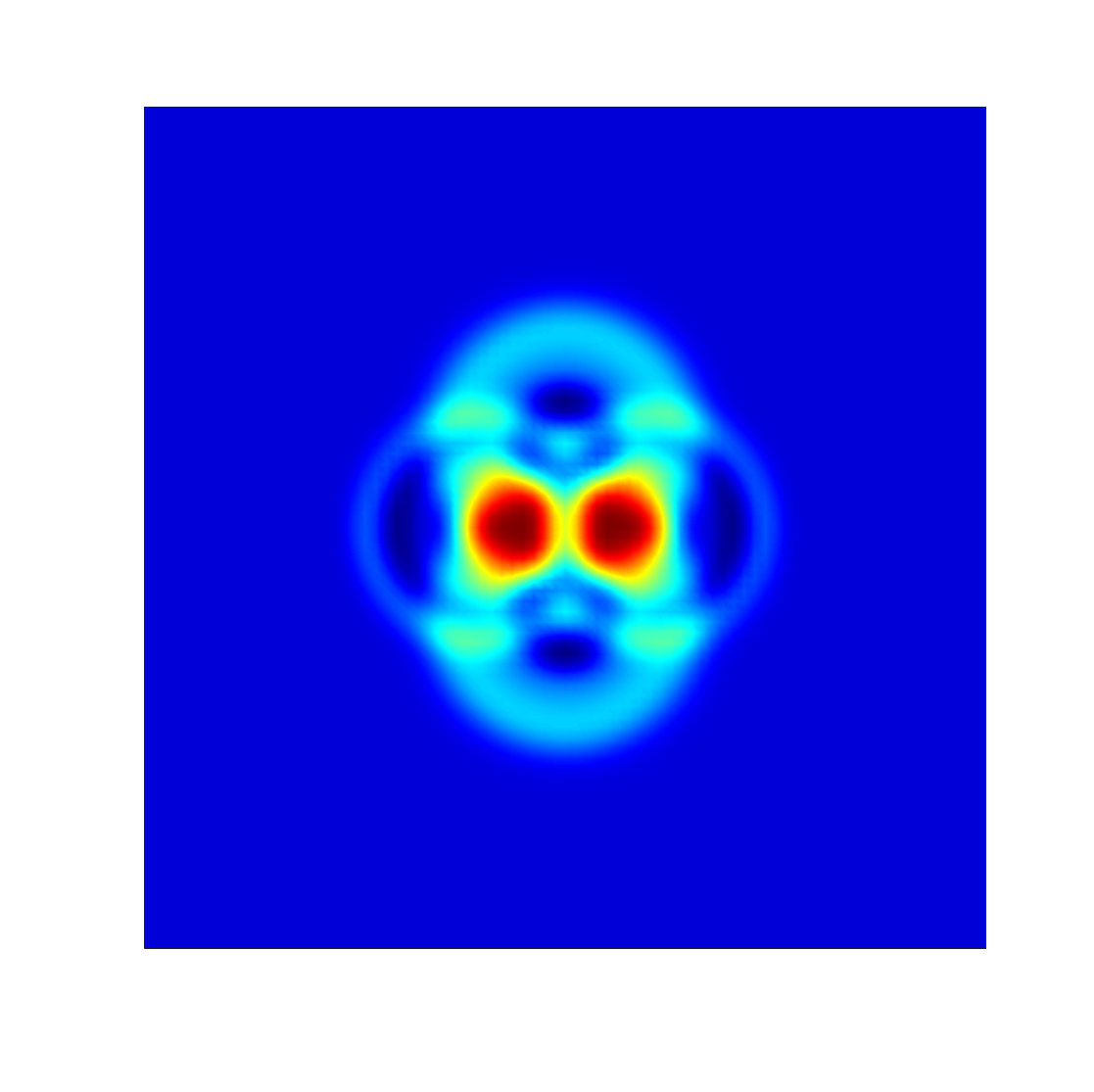}
		}
	\subfigure[$t=3$]
	{
		\includegraphics[width=0.25\linewidth]{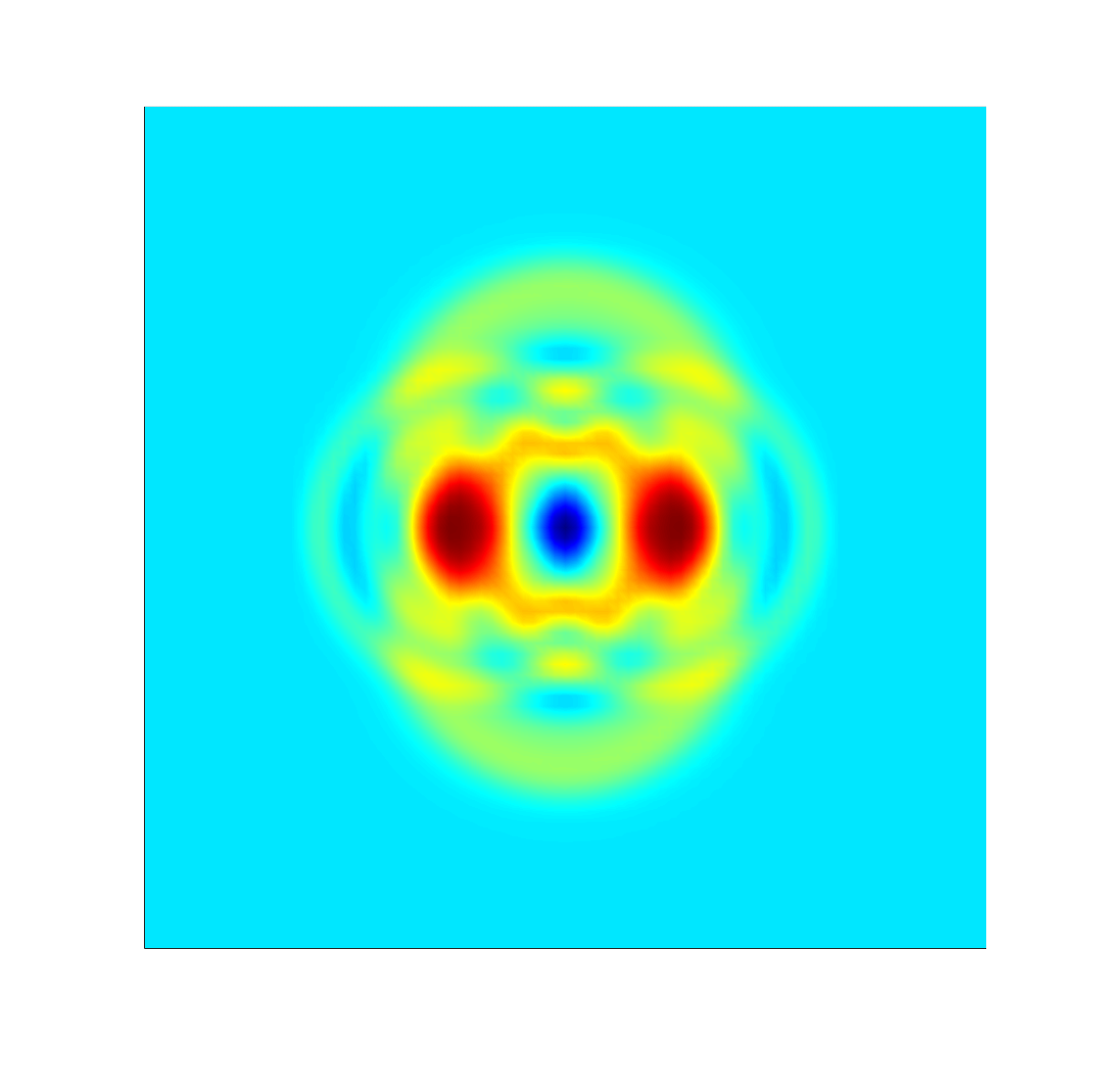}
		}
	\subfigure[$t=4$]
	{
		\includegraphics[width=0.25\linewidth]{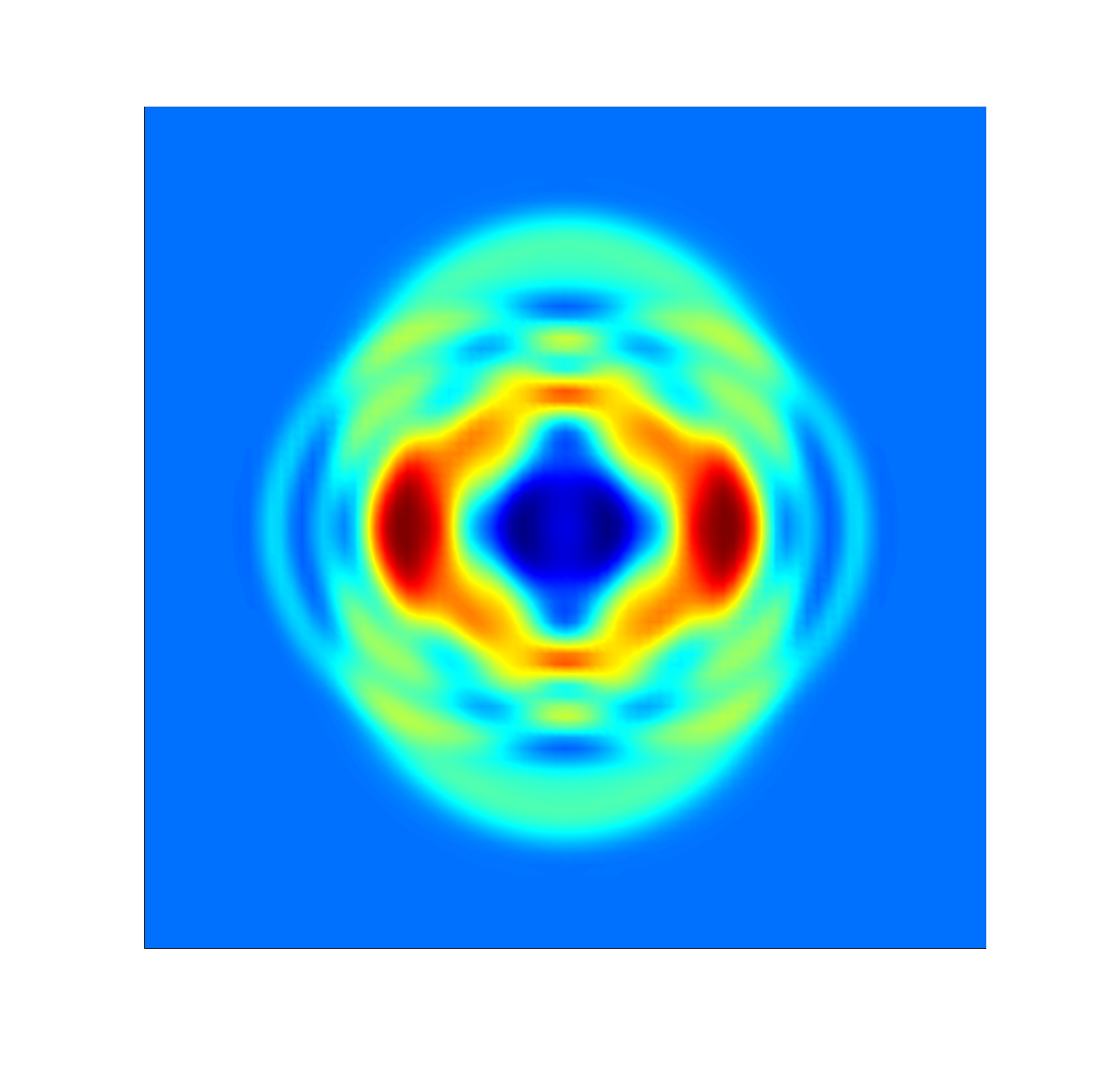}
		}
	\subfigure[$t=5$]
	{
		\includegraphics[width=0.25\linewidth]{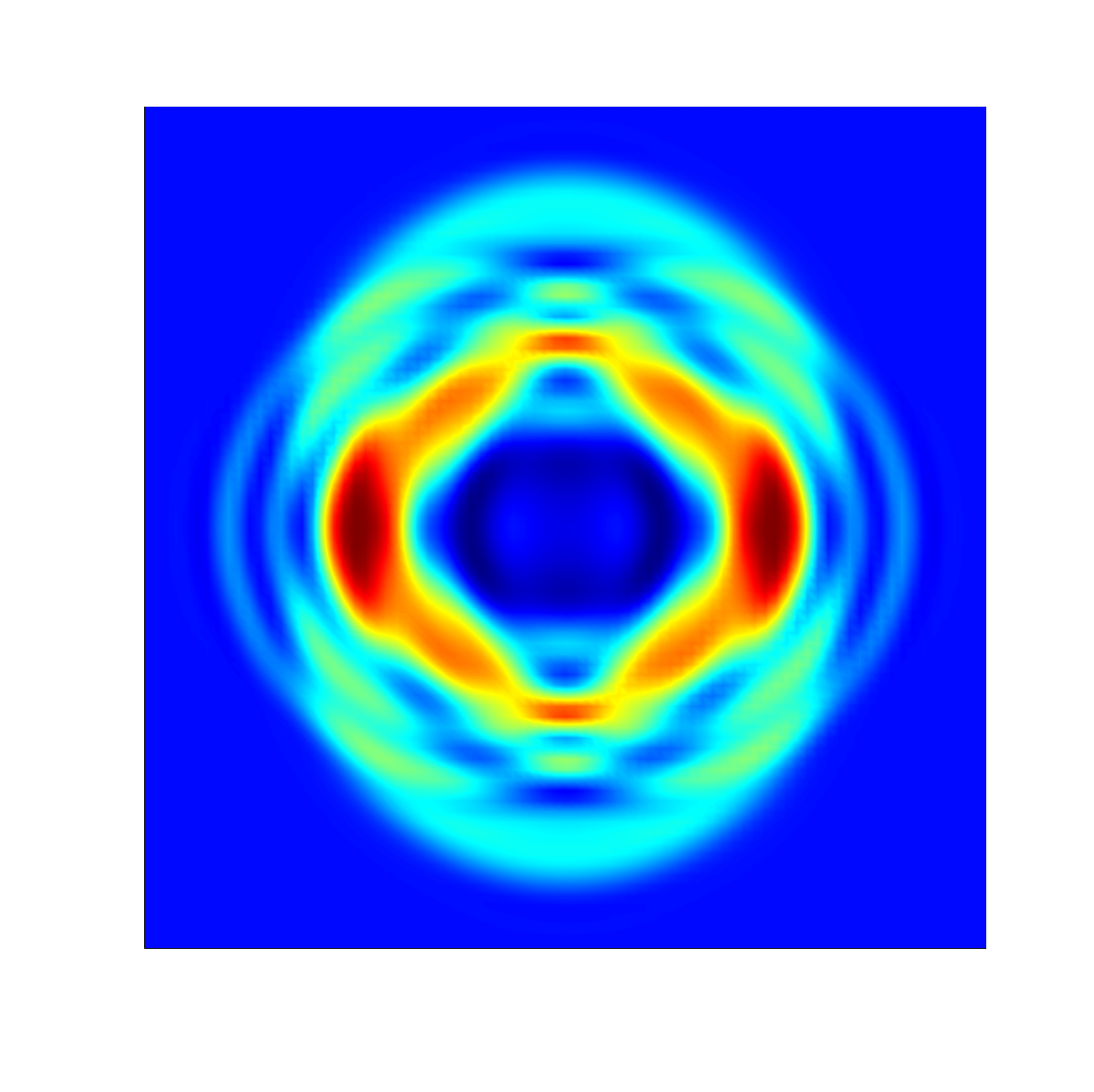}
		}
		}
	\caption{\ Temporal evolution of the approximate solution $\varphi$ depicted via contour lines.}
	\label{fig4.4-2}
\end{figure}

\subsection{Three-dimensional case}
\begin{example}\label{example4.2.1}
(Convergence test) The exact solutions are given by
\begin{equation}
	\renewcommand{\arraystretch}{1.5}
	\left\{\begin{array}{llllll}
		u=e^{-2t}\sin(\pi x)\sin(\pi y)sin(\pi z)+ie^{-3t}\sin(\pi x)\sin(\pi y)\sin(\pi z),\nonumber\\
		\varphi=3\pi^2e^{-t}\sin(\pi x)\sin(\pi y)\sin(\pi z).\nonumber
	\end{array}
	\right.
\end{equation}
\end{example}

Let $\Omega = (0,1) \times (0,1) \times (0,1)$ and $T = 1.0$. A structured Cartesian grid is employed to partition $\Omega$ in Fig. \ref{fig4.5}, with $M$ equal intervals along each coordinate direction; the resulting spatial step size is therefore $h = 1/M$. Guided by the a priori error bounds derived in Section 3, the temporal discretization parameter is set to $\tau = h$, guaranteeing that the time-step refinement aligns with the spatial mesh refinement in the asymptotic limit. At the final time $t = T$, we evaluate six discrete error measures:  
$\|I_h u^n - u_h^n\|_1$, $\|I_{2h} u_h^n - u^n\|_1$, $\|I_h \phi^n - \phi_h^n\|_1$, $\|I_{2h} \phi_h^n - \phi^n\|_1$, $\|p_h^n - p^n\|_0$, and $\|\varphi_h^n - \varphi^n\|_0$,  
all of which are reported  in Table \ref{table4.2}. 

\begin{figure}[htbp]
	\centering
	\includegraphics[width=0.4\textwidth]{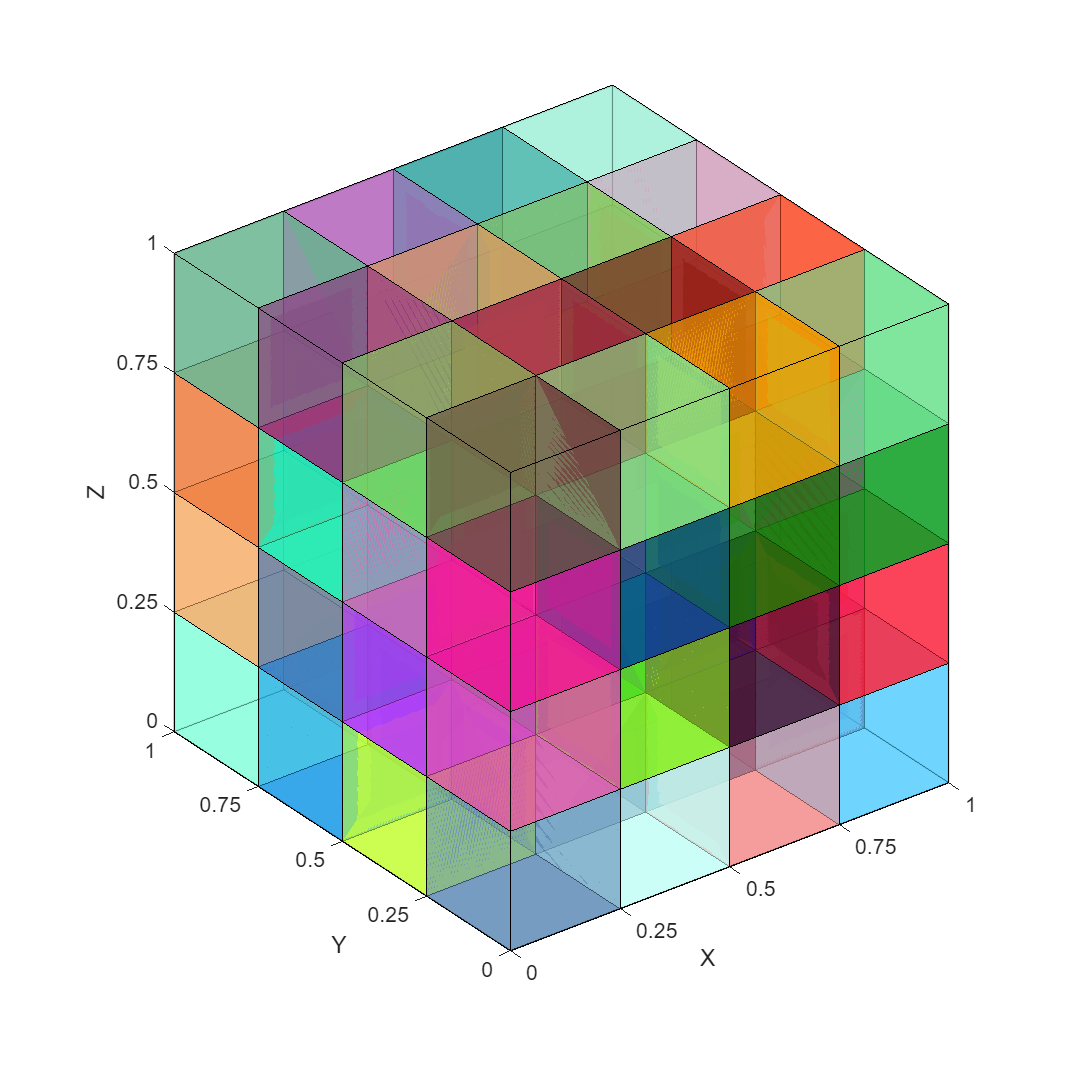}
	\caption{\ Mesh.}
	\label{fig4.5}
\end{figure}

\begin{table}[htbp]
	\caption{Errors  and rates at $t = 1.0$ for 3D case}\label{table4.2}
	\centering
	\begin{tabular*}{\textwidth}{@{\extracolsep{\fill}}l|llll}
		\toprule
		$M \times M\times M$ & $4 \times 4\times 4$ & $8 \times 8\times 8$ & $16 \times 16\times 16$ & $32 \times 32\times 32$  \\
		\midrule
		$\|I_hu^n - u_h^n\|_1$       & 3.5224e-01 & 8.4498e-02 & 2.0723e-02 & 5.1523e-03        \\
		order                        & -          & 2.0596     & 2.0277     & 2.0079            \\\hline
		$\|I_{2h} u_h^n - u^n\|_1$   & 4.0643e-01 & 8.9035e-02 & 2.1263e-02 & 5.2502e-03        \\
		order                        & -          & 2.1906     & 2.0660     & 2.0179            \\\hline
		$\|I_h\phi^n - \phi_h^n\|_1$ & 1.0608e+00 & 2.6191e-01 & 6.5182e-02 & 1.6276e-02        \\
		order                        & -         & 2.0180     & 2.0065     & 2.0018            \\\hline
		$\|I_{2h}\phi_h^n-\phi^n\|_1$& 1.1731e+00 & 2.7361e-01 & 6.6739e-02 & 1.6577e-02        \\
		order                        & -          & 2.1002     & 2.0355     & 2.0093            \\\hline
		$\|p_h^n-p^n\|_0$            & 3.1679e-01 & 1.0647e-01 & 2.8534e-02 & 7.2568e-03        \\
		order                        & -          & 1.5731     & 1.8997     & 1.9753            \\ \hline
		$\|\varphi_h^n-\varphi^n\|_0$& 1.0801e+00 & 2.9473e-01 & 7.5166e-02 & 1.8883e-02        \\
		order                        & -          & 1.8737     & 1.9712     & 1.9930            \\
		\bottomrule
	\end{tabular*}
\end{table}
Numerical results exhibit convergence behavior at the theoretically predicted rate of order two in $h$, in full accordance with the estimates stated in Remark \ref{rem2}. These findings substantiate the extension of the method’s theoretical foundation to the three-dimensional setting of the KGZ equations (\ref{eq1.1}). The decay profiles of these errors are illustrated in Fig. \ref{fig4.6} for graphical verification.
\begin{figure}[htbp]
\centering{
\subfigure[$u_h^n$, $\phi_h^n$.]{\includegraphics[width=0.35\linewidth]{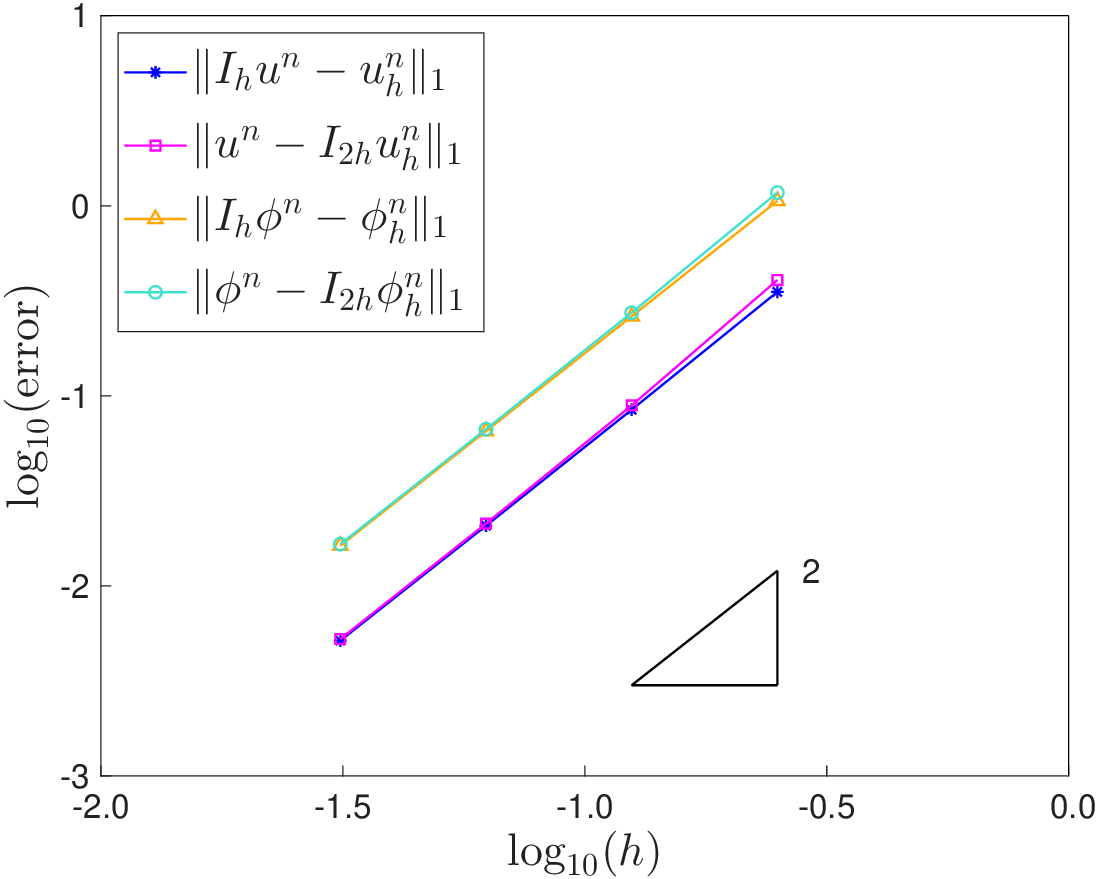}}\hspace{1cm}
\subfigure[$p_h^n$, $\varphi_h^n$.]{\includegraphics[width=0.35\linewidth]{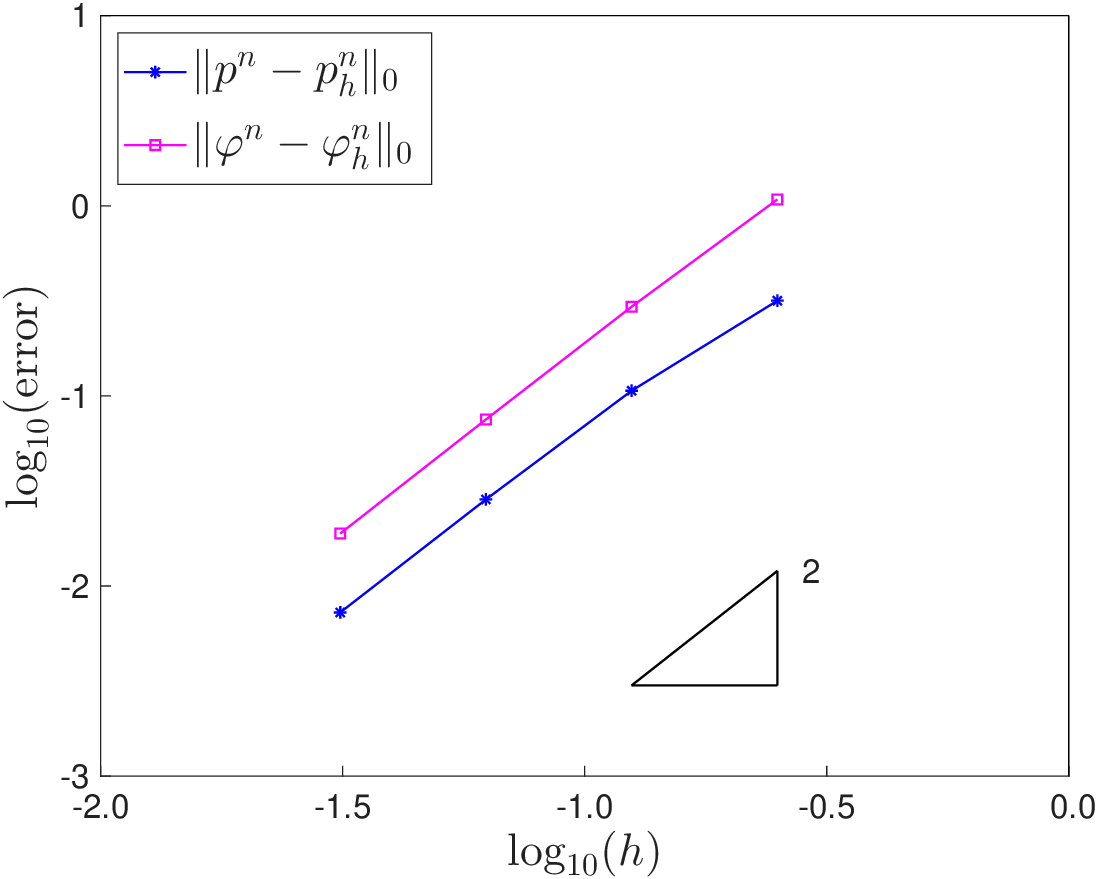}}}
\caption{\ Error curves of $u_h^n$, $\phi_h^n$, $p_h^n$, and $\varphi_h^n$ for 3D case.}
\label{fig4.6}
\end{figure}

\begin{example}\label{example4.2.2}
	(Dynamics of wave interactions) The initial conditions are given by
	\begin{equation}
	\renewcommand{\arraystretch}{1.5}
	\left\{\begin{array}{llllll}
		u_0=(1+\frac{i}{2})[\text{exp}(-(x+2)^2-y^2-z^2)+\text{exp}(-(x-2)^2-y^2)-z^2],\\
		\varphi_0=\text{sech}(x^2+(y+2)^2+z^2)+\text{sech}(x^2+(y-2)^2+z^2),\nonumber\\ 
		u_1=(1+\frac{i}{2})\text{exp}(-x^2-y^2-z^2),\\
		\varphi_1=\text{sech}(x^2+y^2+z^2).
	\end{array}
	\right.
\end{equation}
\end{example}

Let $\Omega=(-5,5)\times(-5,5)\times(-5,5)$,   $h = 5/32$, and $\tau = 0.05$. Isosurface visualizations of $|u|$, $Im(u)$, $Re(u)$ and $\varphi$ at different times ( $t = 0$, $0.3$, $1.2$, and $1.5$) are presented in Figs. \ref{fig4.7}--\ref{fig4.10}.  These numerical results illustrate that the developed numerical scheme is capable of robustly capturing the intrinsic wave propagation behaviors, nonlinear coupling effects, and complex interaction mechanisms governed by the three-dimensional KGZ system (\ref{eq1.1}), whilst simultaneously maintaining prominent numerical accuracy and satisfactory computational efficiency.

\begin{figure}[H]
	\centering{
	\subfigure[$t=0$]
	{
		\includegraphics[width=0.2\linewidth]{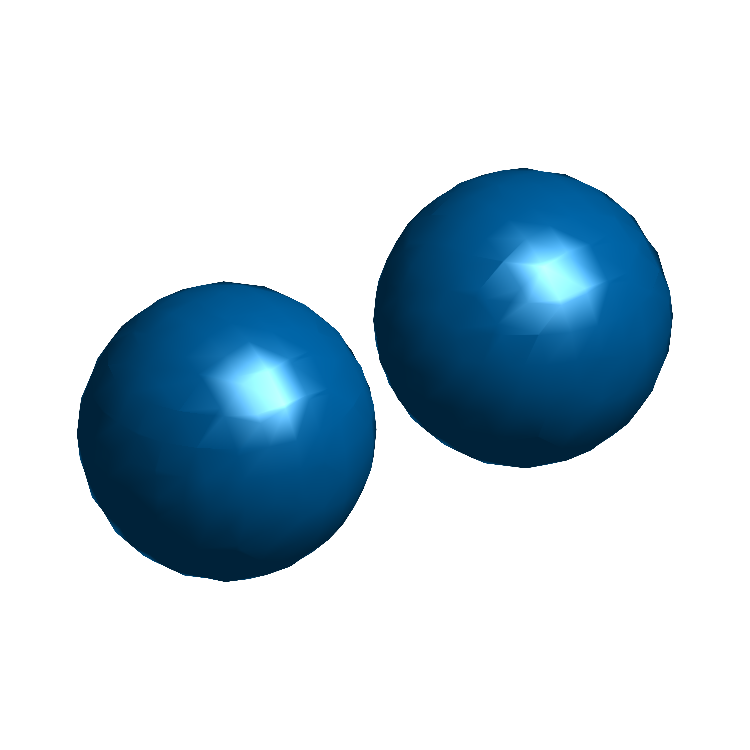}
		}
	\subfigure[$t=0.3$]
	{
		\includegraphics[width=0.2\linewidth]{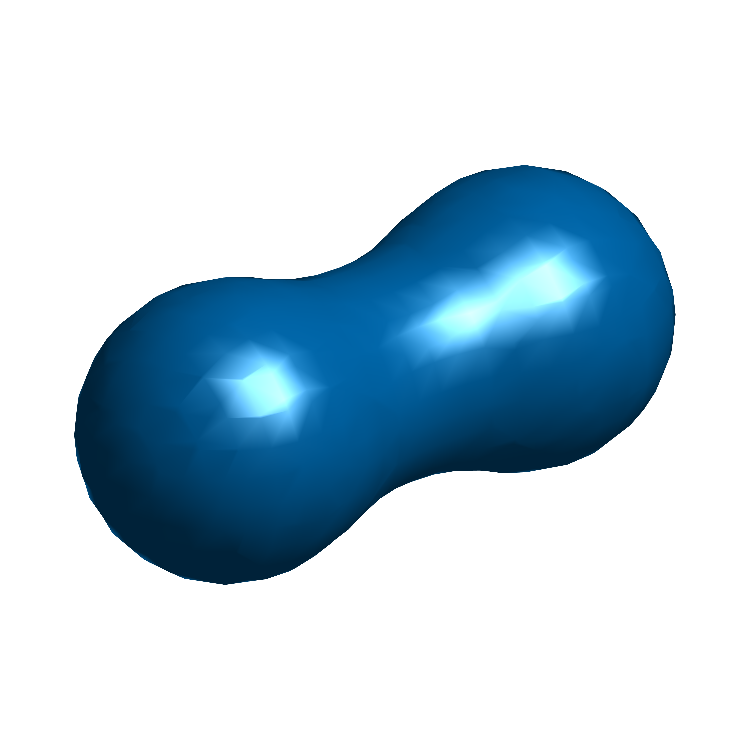}
		}
	\subfigure[$t=1.2$]
	{
		\includegraphics[width=0.2\linewidth]{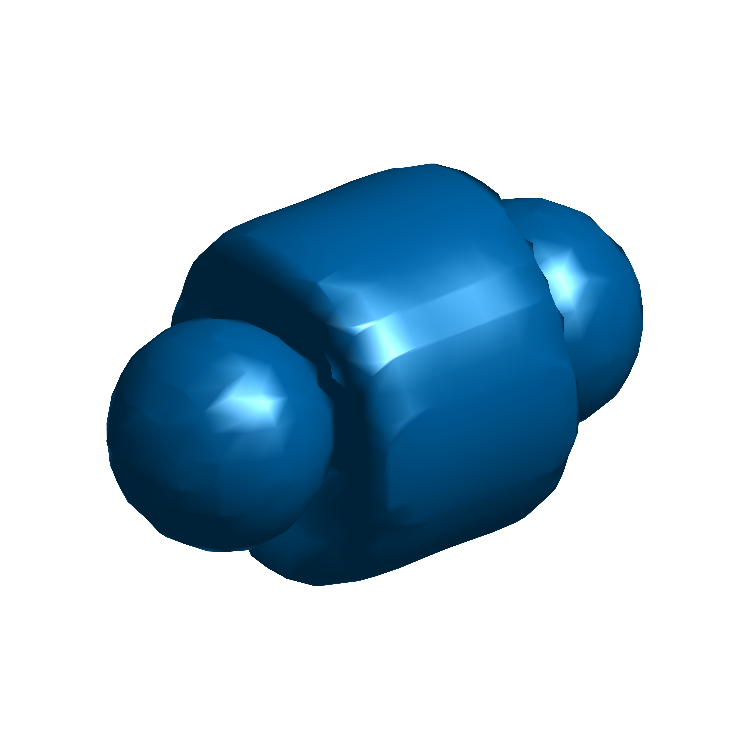}
		}
	\subfigure[$t=1.5$]
	{
		\includegraphics[width=0.2\linewidth]{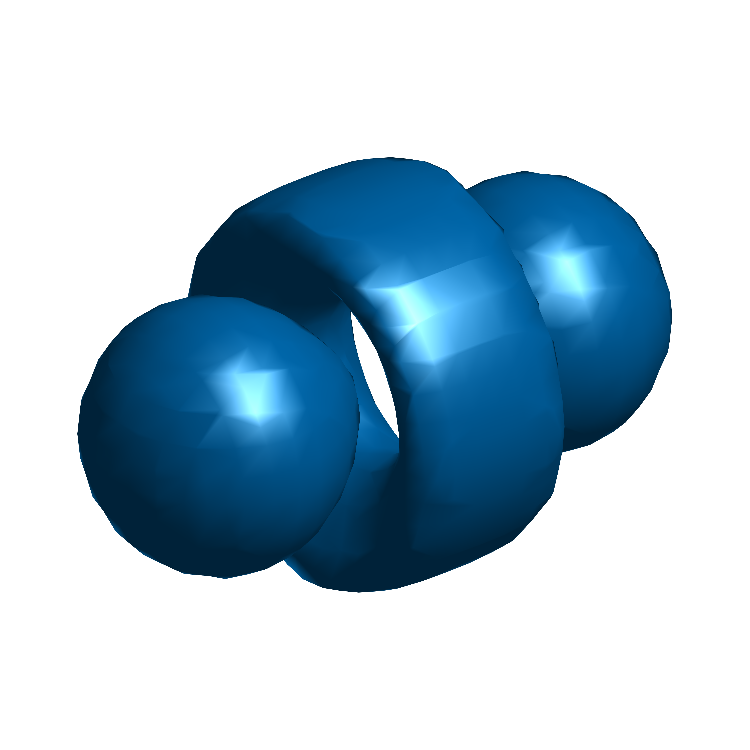}
		}
	}
	\caption{\ Isosurface profiles of $|u|$ at various time.}
	\label{fig4.7}
\end{figure}

\begin{figure}[H]
	\centering{
	\subfigure[$t=0$]
	{
		\includegraphics[width=0.2\linewidth]{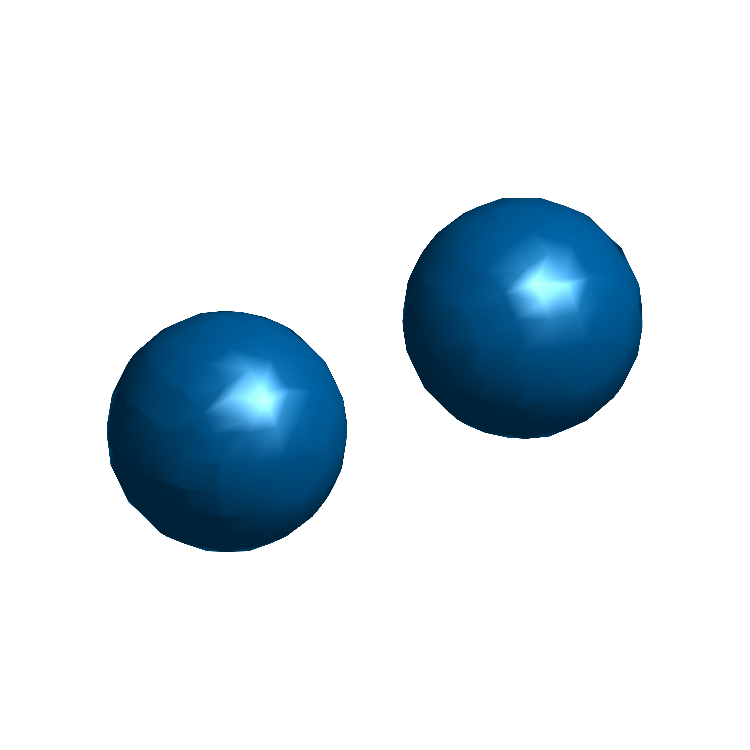}
		}
	\subfigure[$t=0.3$]
	{
		\includegraphics[width=0.2\linewidth]{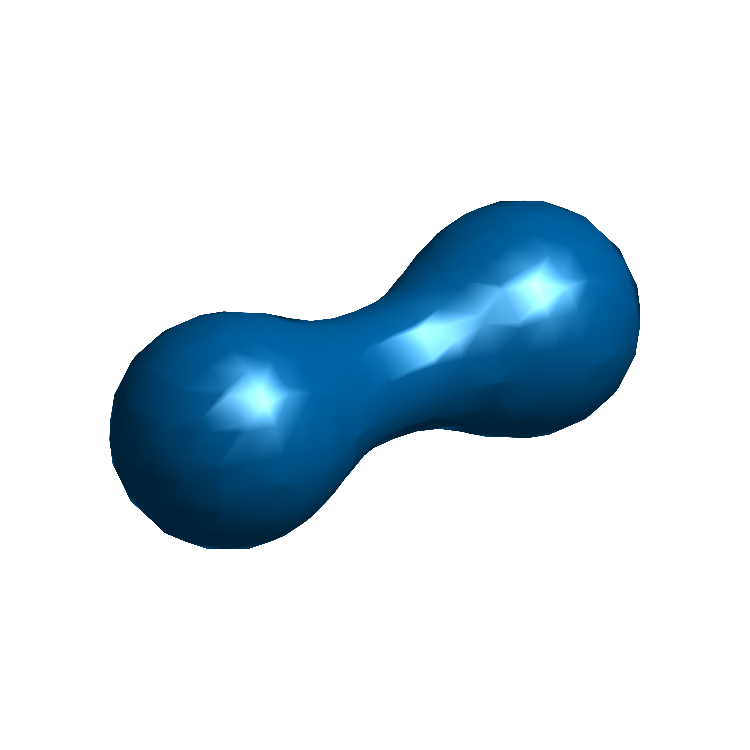}
		}
	\subfigure[$t=1.2$]
	{
		\includegraphics[width=0.2\linewidth]{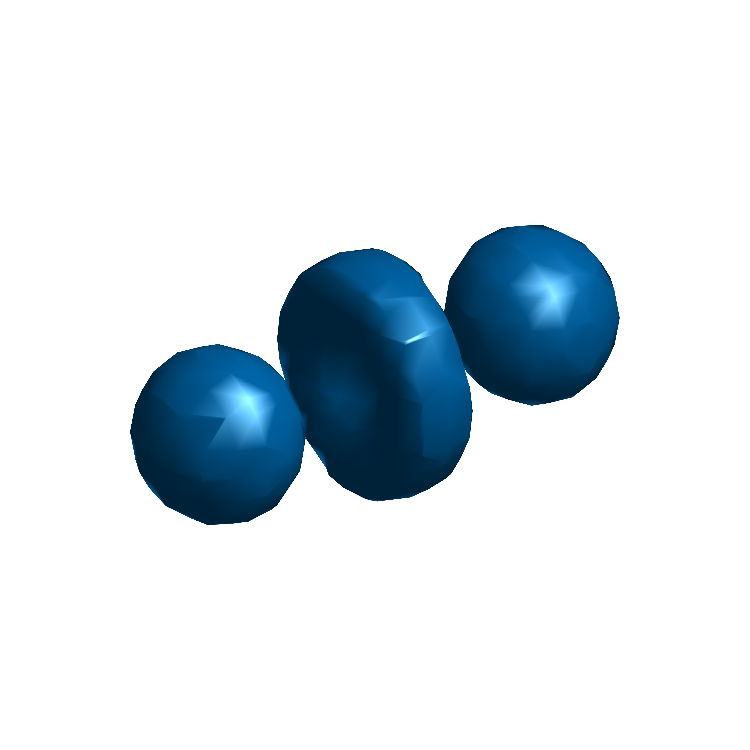}
		}
	\subfigure[$t=1.5$]
	{
		\includegraphics[width=0.2\linewidth]{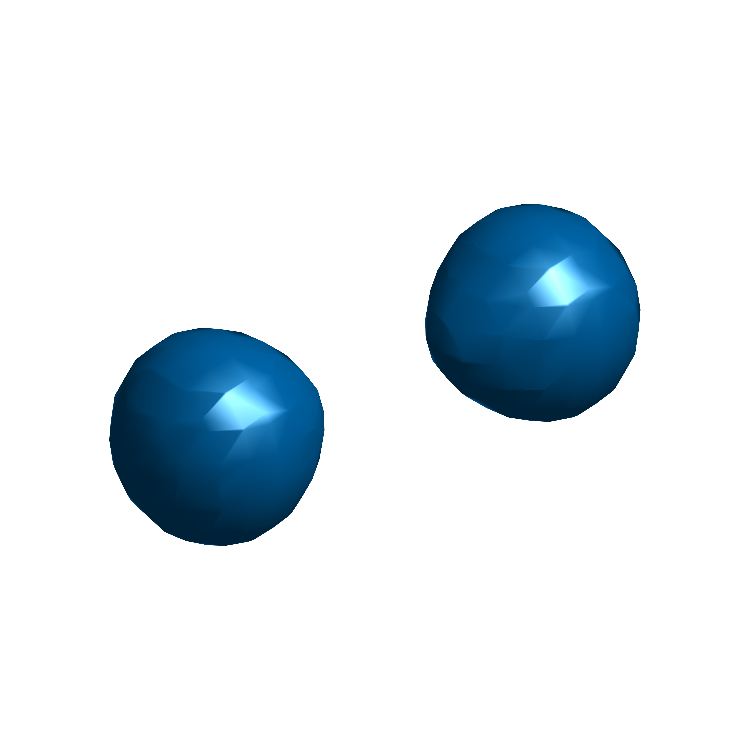}
		}}
	\caption{\ Isosurface profiles of $Im(u)$ at various time.}
	\label{fig4.8}
\end{figure}

\begin{figure}[H]
	\centering{
	\subfigure[$t=0$]
	{
		\includegraphics[width=0.2\linewidth]{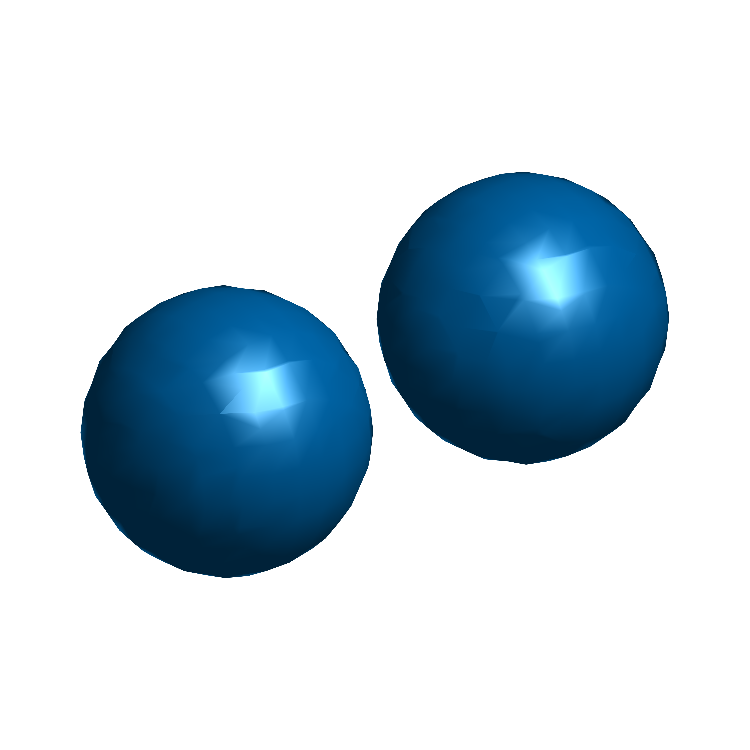}
		}
	\subfigure[$t=0.3$]
	{		\includegraphics[width=0.2\linewidth]{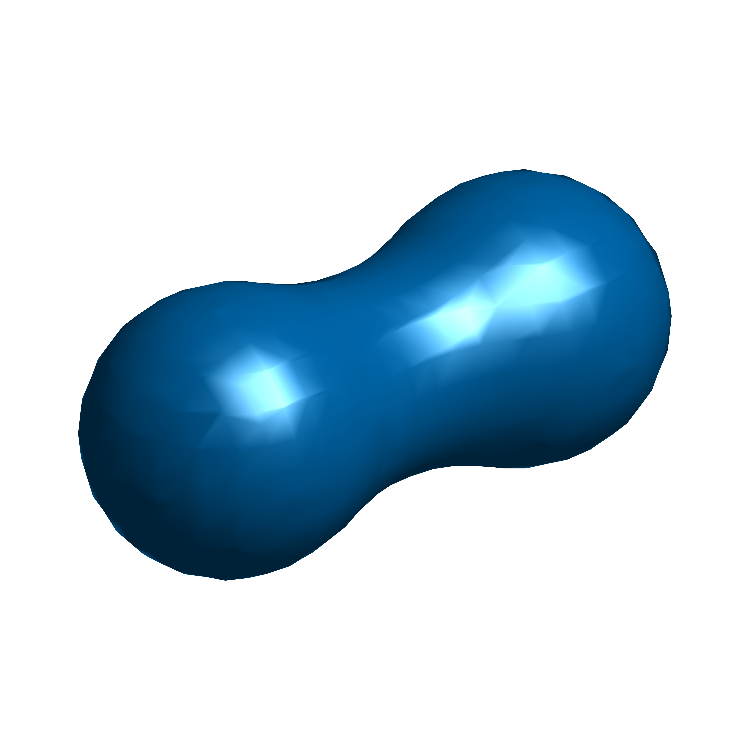}
		}
	\subfigure[$t=1.2$]
	{
		\includegraphics[width=0.2\linewidth]{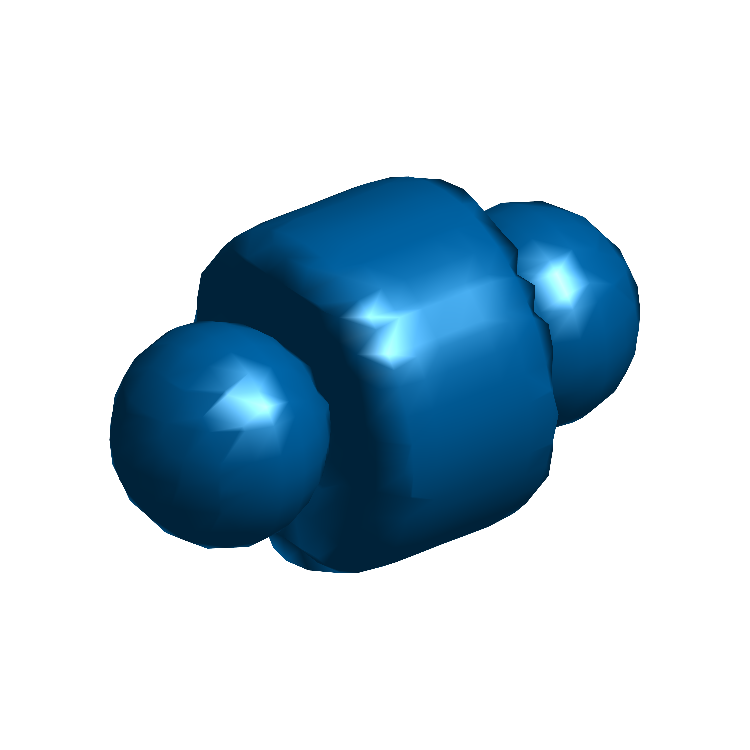}
		}
	\subfigure[$t=1.5$]
	{
		\includegraphics[width=0.2\linewidth]{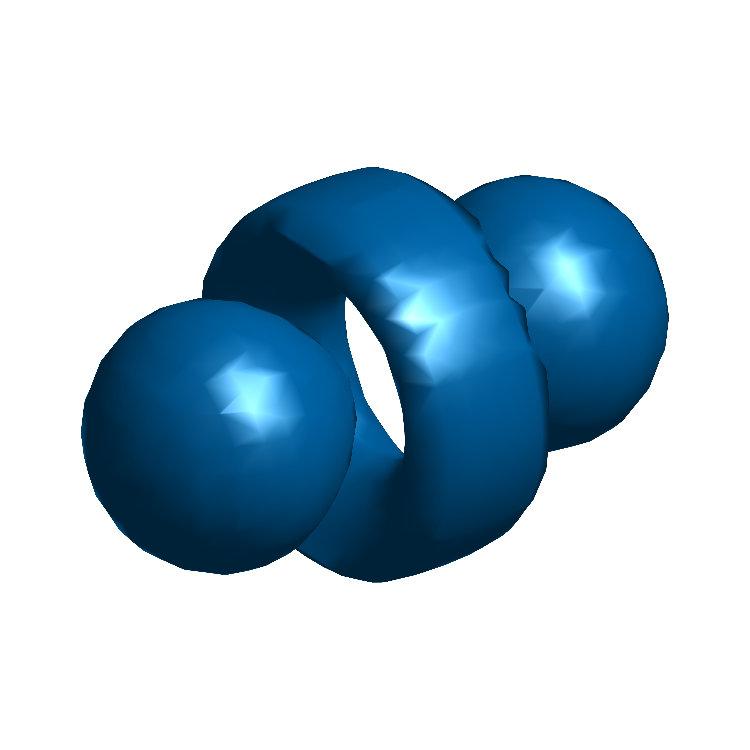}
		}
	}
	\caption{\ Isosurface profiles of $Re(u)$ at various time.}
	\label{fig4.9}
\end{figure}

\begin{figure}[H]
	\centering{
	\subfigure[$t=0$]
	{
		\includegraphics[width=0.2\linewidth]{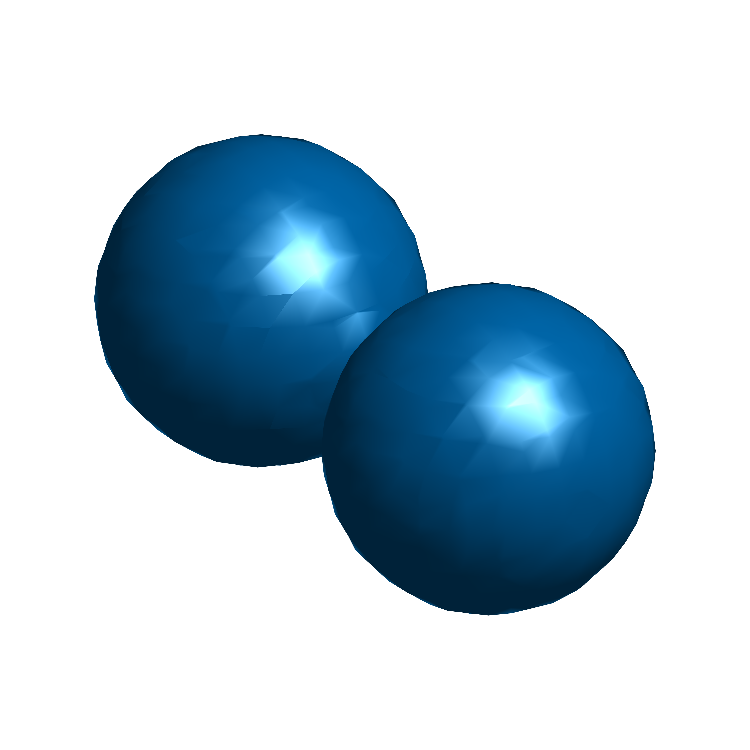}
		}
	\subfigure[$t=0.3$]
	{
		\includegraphics[width=0.2\linewidth]{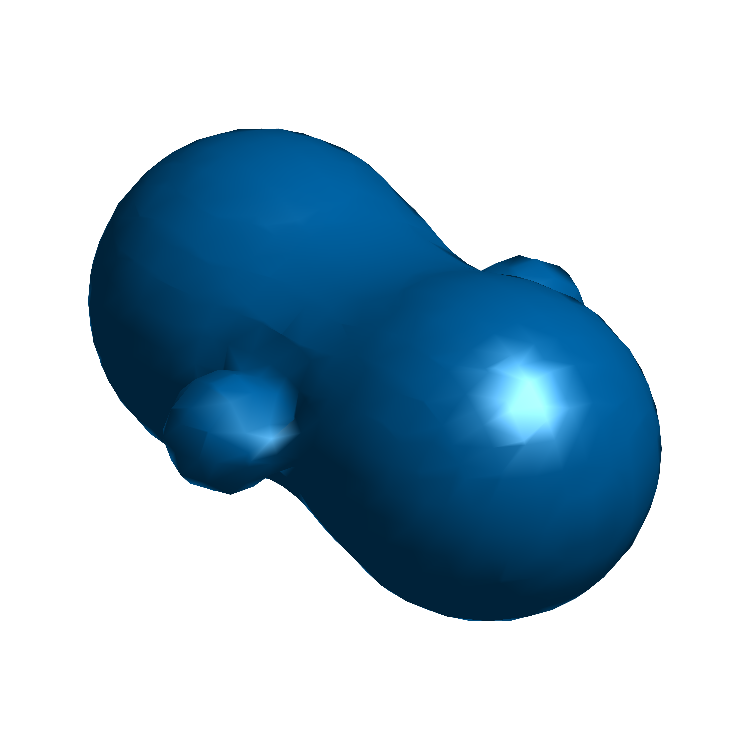}
		}
	\subfigure[$t=1.2$]
	{
		\includegraphics[width=0.2\linewidth]{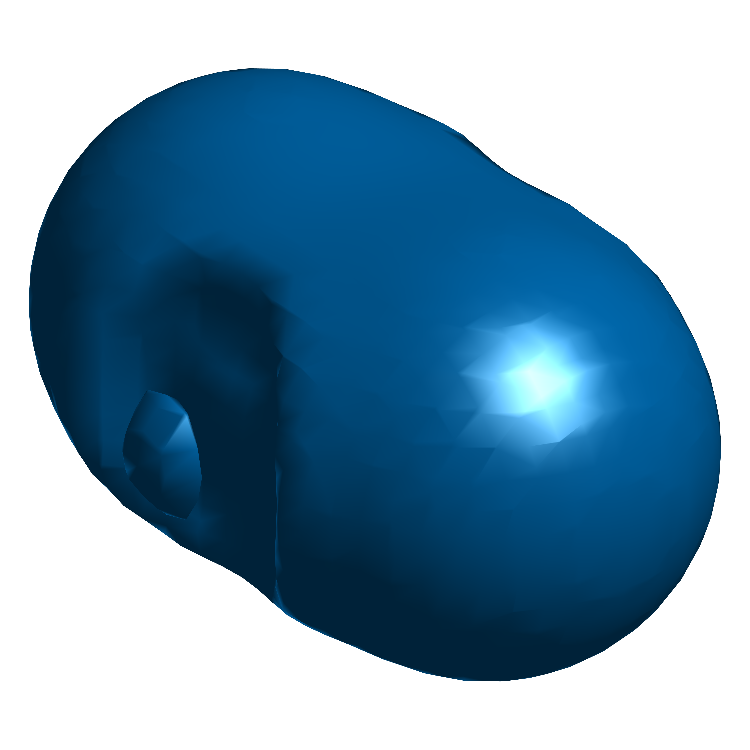}
		}
	\subfigure[$t=1.5$]
	{
		\includegraphics[width=0.2\linewidth]{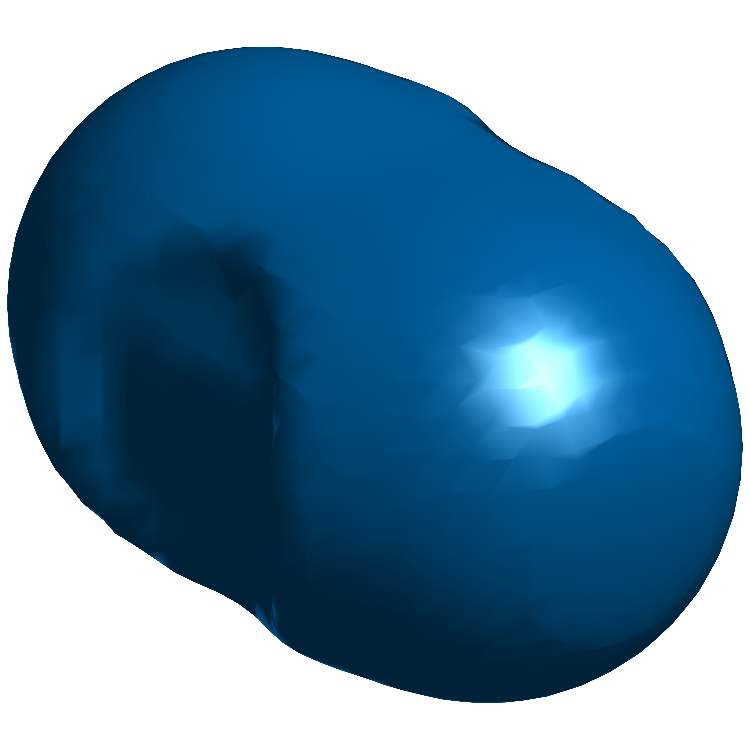}
		}
	}
	\caption{\ Isosurface profiles of $\varphi$ at various time.}
	\label{fig4.10}
\end{figure}

\section{Conclusion}
This article develops a rigorous theoretical and numerical framework for the generalized KGZ equations based on an energy-conservation bilinear Galerkin finite element method. A central contribution is the derivation of $H^1$-norm superconvergence for the primary variables $u$ and $\phi$ under relaxed regularity conditions, where the integrated Ritz projection and interpolation techniques serve as the critical analytical tool. Meanwhile, optimal $L^2$-norm error estimates are established for the auxiliary variables $p$ and $\varphi$. Numerical experiments quantitatively corroborate all theoretical analysis, validating both the high-order accuracy and the discrete energy-conservation properties of the proposed method.  

\section*{Acknowledgements}
Zhang's work was supported partially by the National Key Research and Development Program of China (2023YF1009003), the Natural Science Foundation of Shandong Province (ZR2023MA081) and the Fundamental Research Funds for the Central Universities of China (26CX03004A). J. Zhu’s work was partially supported by the National Council for Scientific and Technological Development of Brazil (CNPq).

\biboptions{sort&compress}
\bibliography{reference}

\end{document}